\documentclass[reqno, 11pt]{amsart}

\usepackage{mathtools}
\usepackage[T1]{fontenc}
\usepackage{amsmath, amsthm, amssymb}
\usepackage[ansinew]{inputenc}
\usepackage{xcolor}
\usepackage{bbm}
\usepackage{mathrsfs}
\usepackage{cite}
\usepackage{xparse}
\usepackage{fullpage}
\usepackage{pdflscape}
\usepackage{stmaryrd}
\usepackage{slashed}
\usepackage{multirow}
\usepackage{lscape}
\usepackage{dsfont}
\usepackage{amsfonts}
\usepackage{scrextend} 
\usepackage{longtable}
\usepackage{verbatim}
\usepackage{subfig}
\usepackage{yfonts}
\usepackage{enumitem}
\usepackage{tabu}
\usepackage[colorlinks=true,allcolors=red]{hyperref}
\usepackage{graphicx} 
\usepackage{accents}

\usepackage{array}
\newcolumntype{P}[1]{>{\centering\arraybackslash}p{#1}}
\newcolumntype{M}[1]{>{\centering\arraybackslash}m{#1}}

\usepackage[left=2.1cm, right=2.1cm]{geometry}

\DeclareSymbolFont{sfletters}{OML}{cmbrm}{m}{it}
\DeclareMathSymbol{\somega}{\mathord}{sfletters}{"21}
\DeclareSymbolFont{Letters} {U}{zeur}{m}{n} 
\DeclareMathSymbol{\ssomega}{\mathalpha}{Letters}{"21} 
\DeclareSymbolFont{Letters} {U}{zeur}{m}{n} 
\DeclareMathSymbol{\pPhi}{\mathalpha}{Letters}{"1E}

\makeatletter
\newcommand*\bigcdot{\mathpalette\bigcdot@{.5}}
\newcommand*\bigcdot@[2]{\mathbin{\vcenter{\hbox{\scalebox{#2}{$\m@th#1\bullet$}}}}}
\makeatother

\newtheorem{theorem}{Theorem}[section]
\newtheorem{lemma}[theorem]{Lemma}
\newtheorem{prop}[theorem]{Proposition}
\newtheorem{cor}[theorem]{Corollary}

\theoremstyle{definition}
\newtheorem{definition}[theorem]{Definition}

\theoremstyle{remark}
\newtheorem{remark}[theorem]{Remark}

\allowdisplaybreaks[4]

\numberwithin{equation}{section}

\usepackage{scalerel,stackengine}
\stackMath
\newcommand\reallywidehat[1]{
\savestack{\tmpbox}{\stretchto{
  \scaleto{
    \scalerel*[\widthof{\ensuremath{#1}}]{\kern.1pt\mathchar"0362\kern.1pt}
    {\rule{0ex}{\textheight}}
  }{\textheight}
}{2.4ex}}
\stackon[-6.9pt]{#1}{\tmpbox}
}
\makeatletter
\newcommand*{\rom}[1]{\expandafter\@slowromancap\romannumeral #1@}
\makeatother

\begin{document}

\title{Null coordinates for quasi-periodic $(1+1)$-dimensional wave operators on the circle with applications to reducibility}

\author{Athanasios Chatzikaleas}
\address{Institut de Math\'ematiques, \'Ecole Polytechnique F\'ed\'erale de Lausanne (EPFL), 1015 Lausanne, Switzerland}
\email{athanasios.chatzikaleas@epfl.ch}

\author{Jacques Smulevici}
\address{Laboratory Jacques-Louis Lions (LJLL), Sorbonne Universit\'e, 4 place Jussieu, 75252 Paris, France}
\email{jacques.smulevici@ljll.math.upmc.fr}

\dedicatory{}

\maketitle

\begin{abstract} 
Given any wave operator with principle part $\partial_{t}^2  -\partial_{x}^2 +\mathcal{B}^{xx}(\omega t,x)\partial_{x}^2$, where $\mathcal{B}^{xx}:\mathbb{T}^{\nu+1} \rightarrow \mathbb{R}$ is a sufficiently small, quasi-periodic perturbation and $\omega \in \mathbb{R}^\nu$, we explain how to construct \emph{null coordinates} that respect the quasi-periodicity of the solutions. As it turns out, in these coordinates, the principal symbol of the wave operator above has constant coefficients. To construct these coordinates, we start by writing the wave operator in geometric form, modulo terms of order $1$, meaning as the wave operator arising from an $(1+1)$-Lorentzian metric and then define null coordinates as solutions to the Eikonal equations, so that the metric is conformally flat in these coordinates. The problem of constructing these coordinates is then eventually reduced to that of straightening a vector field with quasi-periodic coefficients. As an application, we give a novel proof of a recent reducibility result of Berti-Feola-Procesi-Terracina for the quasi-periodically forced linear Klein-Gordon equation, the novelty concerning essentially the analysis of the maximal order terms. In particular, the method we propose here does not rely on any quantitative Egorov-type result. 
\end{abstract}
\tableofcontents 
\addtocontents{toc}{\protect\setcounter{tocdepth}{1}} 
\section{Introduction}
\noindent 
Consider a linear Klein-Gordon equation on the circle $\mathbb{S}^1:=\mathbb{R} / \left(2\pi \mathbb{Z} \right)$ of the form 
\begin{equation}\label{eq:KG}
\partial_{t}^2 \psi -\partial_{x}^2 \psi +\mathtt{m} \psi + B^{xx} \partial_{x}^2 \psi  +  B^x \partial_{x} \psi  + B^t \partial_{t} \psi  + B \psi=0\, , 
\end{equation} 
where $\mathtt{m} \ge 0$ and the coefficients\footnote{The top indices on the $B$ coefficients, such as $xx$ on $B^{xx}$, are just labels to remind the reader to which terms in the equation they correspond to. On the other hand, we will use the standard $\partial$ notation to denote partial derivatives.} $B^{xx}$, $B^x$, $B^t$, $B$ are all real valued, smooth functions of $(t,x) \in \mathbb{R} \times \mathbb{S}^1$, which are small in an appropriate regularity class and are \emph{quasi-periodic} with respect to $t$: there exist smooth real valued functions $\mathcal{B}^{xx}$, $\mathcal{B}^x$, $\mathcal{B}^t$, $\mathcal{B}$, all defined on $\mathbb{T}^\nu \times \mathbb{S}^1:= \mathbb{T}^{\nu+1}$, $\nu \in \mathbb{N}^\star$, and a frequency vector $\omega \in \mathbb{R}^\nu$ such that 
\begin{align}
B^{xx}(t,x) &= \mathcal{B}^{xx}(\omega t, x), &B^{x}(t,x) &= \mathcal{B}^{x} (\omega t, x),  \quad 
B^{t}(t,x) = \mathcal{B}^{t}(\omega t, x), &B(t,x) &= \mathcal{B}(\omega t, x).  \label{def:qpc12}   
\end{align}
The \emph{reducibility} problem\footnote{See for instance the introduction of \cite{berti2024reducibility} for a concise historical review of the reducibility problem.} for \eqref{eq:KG} consists essentially of the construction of a quasi-periodic transformation of the phase space that allows one to transform \eqref{eq:KG} into a constant coefficients equation. Typically, reducibility has important corollaries such as Sobolev stability i.e.~the fact that the $H^s$ norms of the solutions do not grow in time. It is also usually considered as an important\footnote{The exception being the results obtained by multi-scale analysis, as in \cite{MR4321994}.} step for the construction of a KAM theory of non-linear PDEs\footnote{See for instance \cite{MR3112201, MR3187681} for examples of reducibility applied to semi-linear wave equations.}. In these applications, Nash-Moser schemes are typically used and as a consequence, it is crucial that the transformation implying reducibility verifies \emph{tame} estimates, so that one needs to control in a, precise, quantitative, way all the operations leading to the quasi-periodic transformation. \\ \\
For the problem \eqref{eq:KG} (with $B^t=0$), reducibility (with tame estimates) was recently obtained in the work of Berti-Feola-Procesi-Terracina \cite{berti2024reducibility} under appropriate discrete symmetries of the coefficients. The main novelties of \cite{berti2024reducibility} concern the analysis of the second order perturbation term $B^{xx} \partial_{x}^2 \psi $, since several previous works had already addressed, in various contexts, the case of perturbations of order\footnote{Note that terms of order $1$ are not maximal perturbations for the second order formulation of the Klein-Gordon equation, as in \eqref{eq:KG}. On the other hand, terms of order $1$ are maximal when we refer to first order formulations.}  1 or 0 of the $1$-dimensional Klein-Gordon equation, such as \cite{MR3112201, MR3187681} for perturbation of order $1$ and \cite{MR3987220} for those of order $0$. See also \cite{MR3603787, MR3932596} for perturbations of maximal orders which however depends only on time. \\ \\
Klein-Gordon or wave type equations can usually be written as a coupled system of transport equations. This is the approach of \cite{berti2024reducibility} where it leads to a system of pseudo-differential\footnote{The pseudo-differential, instead of differential, operators arise naturally since the standard Klein-Gordon operator factorises as $\left(i\partial_t + D_{\mathtt{m}}\right) \left(i\partial_t - D_{\mathtt{m}}\right)$, where $D_{\mathtt{m}}= \mathrm{Op}\left(\sqrt{|\xi|^2+\mathrm{m}}\right)$ is a Fourier multiplier. The use of $\partial_t \pm D_{\mathtt{m}}$, compared to $\partial_t \pm \partial_x$, to rewrite the wave equation, has several advantages. In particular,  $D_{\mathtt{m}}$ respects the $x$-parity of functions and already takes into account the mass term in the Klein-Gordon equation \eqref{eq:KG}.}  type equations of order $1$ with non-constant coefficients. The first step of the reducibility is to understand how to reduce the top order terms in the pseudo-differential operators of order $1$. In the case of a differential (as opposed to pseudo-differential) transport operator, this problem has been recently studied in \cite{reducimoser}. The analysis of \cite{berti2024reducibility} relies on \cite{reducimoser} together with a quantitative novel Egorov type analysis, which concerns the transformation of pseudo-differential operators under change of coordinates and is important to address the additional pseudo-differential aspects of the equations.\\ \\
One of the aims of the present work is to provide a different proof of the reducibility for \eqref{eq:KG}, especially for the top order terms, which in particular does not rely on the Egorov analysis, thus simplifying the original proof of \cite{berti2024reducibility}. Moreover, in our opinion, the approach that we present below is less reliant on the exact nature of the eigenfunctions of the unperturbed problem, a property we hope that will be useful for further applications. For instance, our reduction of the maximal terms do not rely on decompositions into positive and negative frequencies and the associated projection operators. To bypass the Egorov analysis, we show how one can reduce the second order perturbation term in \eqref{eq:KG} directly at the level of the second order formulation of the equation, for which everything is clearly differential and not pseudo-differential. To this end, we rely on the construction of special coordinates, well-known in General Relativity and Lorentzian geometry, called \emph{null coordinates}. These are obtained by solving the Eikonal equation, which in one dimension are just regular linear transport equations. The difficulty lies in constructing such null coordinates which respect the quasi-periodicity of the equations, since quasi-periodic functions with respect to the original coordinates should still be quasi-periodic with respect to the new coordinates. To this end, we show how one can essentially reduce this problem to the one already treated in \cite{reducimoser}, that of straightening standard vector fields on the torus. Since the way we address the maximal order terms does not rely on pseudo-differential operators but in relatively simple transformations, it also leads to some simplifications to the treatment of the lower order terms. Thus, we also provide in this article a reducibility of the terms of order $1$ in \eqref{eq:KG}.  

\subsection{Null coordinates and the main theorems}
The main novelty of this paper therefore lies in the construction of these null coordinates. This is best understood in the context of \emph{$(1+1)$-dimensional Lorentzian geometry}.  Consider thus a Lorentzian metric $g$ on the cylinder $\mathbb{R}_t\times \mathbb{S}^1_x$ with coordinates $(t,x)$ that takes the form
\begin{align}\label{eq:gmetric} 
g &=\sum_{\alpha, \beta=0}^1 g_{\alpha \beta}\, dx^\alpha \otimes dx^\beta= -A dt\otimes dt + A^{-1} dx \otimes dx,   
\end{align}
where
\begin{align*}
	A:= A(t,x) >0, \quad dx^0=dt, \quad dx^1=dx. 
\end{align*}
To any such metric, we associate its wave operator $\square_g$  acting on scalar functions $\psi: \mathbb{R}_t\times \mathbb{S}^1_x \rightarrow \mathbb{R}$, defined in coordinates by
\begin{align}\label{eq:waveop}
\square_g \psi: &= \frac{1}{\sqrt{\left| \mathrm{det\, g} \right|}}\sum_{\alpha, \beta}\partial_{\alpha} \left( \left(g^{-1}\right)^{\alpha \beta}\sqrt{\left| \mathrm{det\, g} \right|} \partial_{\beta} \psi \right) 
= - \partial_t \left( A^{-1}  \partial_t \psi \right) +  \partial_x \left( A  \partial_x \psi \right)\nonumber \\
&=\frac{1}{A}\left(  - \partial_{t}^2 \psi  + A^2 \partial_{x}^2 \psi   + \frac{ \partial_{t} A }{A}\partial_{t} \psi  + A  \partial_{x} A  \partial_{x} \psi \right).
\end{align}
Choosing $A^2=1-B^{xx}$ and assuming that $\psi$ solves \eqref{eq:KG}, we see that  
\begin{align}\label{eq:KGgeom}
\left(\square_g - \mathtt{m} \right) \psi &= \frac{1}{A}\left(  - \partial_{t}^2 \psi  + A^2 \partial_{x}^2 \psi   + \frac{ \partial_{t} A }{A}\partial_{t} \psi  + A  \partial_{x} A  \partial_{x} \psi \right) - \mathtt{m} \psi \nonumber \\
&=  \frac{1}{A}\left(  B^x \partial_{x} \psi  + B^t \partial_{t} \psi  + B \psi +\mathtt{m}\psi + \frac{ \partial_{t} A }{A}\partial_{t} \psi  +  A  \partial_{x} A  \partial_{x} \psi \right) - \mathtt{m} \psi \nonumber \\
&= \left( \frac{B^x+ A   \partial_{x} A }{A} \right) \partial_{x} \psi  + \left( \frac{B^t}{A} + \frac{ \partial_{t} A }{A^2}\right) \partial_{t} \psi  + \left(\frac{B}{A}- \mathtt{m}\left( 1- \frac{1}{A} \right) \right)\psi. 
\end{align}
Hence, we see that \eqref{eq:KG} can be recast as a geometric Klein-Gordon equation with extra (non-geometric) terms which are of order 1 or 0.  \\ \\
Here, \emph{geometric} means that $\square_g - \mathtt{m}$ is an operator associated to the Lorentzian manifold $\left(\mathbb{R}_t\times \mathbb{S}^1_x, g \right)$ and so, by definition, is invariant under any change of coordinates. More precisely, if $\psi: \mathbb{R}_t\times \mathbb{S}^1_x \rightarrow \mathbb{R}$ is a scalar function  and 
\begin{align*}
\Phi: U \subset \mathbb{R}_t\times \mathbb{S}^1_x  \rightarrow  \Phi(U) \subset \mathbb{R}^2 ,\quad  (t,x)  \mapsto \Phi(t,x):=(t',x')
\end{align*}
is a local coordinate system on $\mathbb{R}_t\times \mathbb{S}^1_x$, possibly different from the canonical $(t,x)$ coordinates, then
\begin{align*}
	\left(\square_{g'} - \mathtt{m} \right) \left(\psi \circ \Phi^{-1}\right)= \left[\left(\square_g - \mathtt{m} \right) \psi\right] \circ \Phi^{-1},
\end{align*} 
where $g'$ denotes the metric $g$ in the new coordinate system $(t',x')$ and $\square_{g'}$ is the wave operator written in the $(t',x')$ coordinates,  given by \eqref{eq:waveop} with $g'$ replacing $g$.   \\ \\
From \eqref{eq:waveop}, we see that we can replace the coefficients in front of the top order terms of \eqref{eq:KG} by constant coefficients provided that we can construct a new coordinate system $(t',x')$ such that $g'$ (which is the metric $g$ written with respect to the coordinate system $(t',x')$) has constant coefficients. This is however too restrictive and one can easily construct Lorentzian metrics which do not possess this property, even locally. What is however true it that any $(1+1)$-dimensional Lorentzian metric is \emph{locally conformally flat}: for any $(t_0, x_0) \in \mathbb{R}_t\times \mathbb{S}^1_x$, there exists local coordinates $(t',x') $ defined in a neighbourhood 
of $(t_0, x_0)$ such that 
\begin{eqnarray}\label{eq:locconflat}
g' = \Omega^{2}\eta,\quad 
\eta:=- dt' \otimes dt' + dx' \otimes dx'
\end{eqnarray}
Here, $\eta $ is the flat metric of Minkowski spacetime in $(1+1)$-dimensions and $\Omega:=\Omega(t',x') >0$ is a scalar function, a \emph{conformal} factor. Moreover, another important property associated to $1+1$ Lorentzian geometry is that the wave operator is \emph{conformally invariant} in $(1+1)$-dimensions meaning that 
\begin{align*}
	\square_{g'} = \frac{1}{\Omega^2}\square_{\eta}=\frac{1}{\Omega^2}\left(- \partial^2_{t'} + \partial^2_{x' } \right).
\end{align*}
Thus, with respect to such $(t',x')$ coordinates, the top order terms in \eqref{eq:KG} have constant coefficients (up to a global multiplication factor). In practice, to construct these coordinates, one actually first constructs \emph{null coordinates} $(u,v)$ which are solutions to the Eikonal equations $g^{-1} \left( d u, d u\right)= g^{-1} \left( d v, d v\right)=0$. Indeed, for the metric \eqref{eq:gmetric}, keeping in mind that the two solutions must be independent to form a coordinate system, the functions $(u,v)$ must then verify (up to exchanging $u$ and $v$)\; 
\begin{align*}
	\partial_{t} u + A \partial_{x} u =0,\\
\partial_{t}v - A \partial_{x}v =0.
\end{align*} 
These are transport equations, so they can easily be solved locally. From these equations, it follows that the metric \eqref{eq:gmetric} with respect to $(u,v)$ coordinates takes the form\footnote{From now on, we do not systematically relabel geometric objects when we change coordinates, since they are the same by definition. For example, we still write $g$ instead of $g'$.} 
\begin{align}\label{eq:metricnull}
g = -\frac{\Omega^2}{2} \left( du \otimes dv + dv \otimes du \right),  \quad \Omega^2= - \frac{1}{A  \partial_{x} u  \partial_{x}v }.
\end{align}
In what follows, we will have $\Omega^2>0$, with $\partial_{x} u < 0$, $\partial_{x}v > 0$. Define now $(\tau, R)$ via 
\begin{eqnarray*}
\tau:=\frac{v+u}{2}, \quad R:=\frac{v-u}{2}.
\end{eqnarray*}
It it easy to check that, in $(\tau, R)$ coordinates, the metric takes the form \eqref{eq:locconflat} with $t'=\tau$ and $x'=R$. \\ \\
Nevertheless, these type of coordinates are a priori only locally defined and, in general, one should not expect that such a change of coordinates $(t,x) \mapsto (t',x'):=\Phi (t,x)$ would in fact preserve the quasi-periodicity of functions and operators. Finally, since our goal is to apply this construction to the reducibility of \eqref{eq:KG}, even if we understand in which sense $\Phi$ respects the quasi-periodicity of the equations, we still need to prove that it verifies tame estimates, in order to obtain the same control of solutions as in \cite{berti2024reducibility}. \\ \\
The existence of such a coordinate system verifying these constraints is included in the statement of the main theorem of this article\footnote{
 Here, and throughout the article, we use the same notations as in \cite{berti2024reducibility} and \cite{reducimoser} so that the reader can  facilitate comparisons. A notable difference is that we use $\iota$ instead of $\tau$ for one of the diophantine parameters,  $\tau$ being reserved in the present article for our new time coordinate.}.


\begin{theorem}[Main theorem with parity] \label{th:redumetric}
Let $\nu \in\mathbb{N}^\star$ and fix $s_0 > (\nu +7)/2$. Let $\mathcal{O}_0$ be a compact set of $\mathbb{R}^\nu$ of positive Lebesgue measure and let $\gamma \in (0,1)$ and $\iota := \nu+3$ be fixed diophantine parameters. 
Then, there exist $s_1 \ge s_0+2 \iota +4$ and $\eta_\star  >0$ such that the following holds. 
Let $\mathcal{A}: \mathbb{T}^{\nu+1} \rightarrow (0,+\infty)$ be a smooth function which is even in $x$, $\mathcal{A}\left( \cdot , x \right) = \mathcal{A}\left( \cdot , -x \right)$, even in $\varphi$, $\mathcal{A}\left( \varphi , \cdot \right) = \mathcal{A}\left( -\varphi , \cdot \right)$ and verifies the smallness condition
\begin{equation}\label{ineq:Asmallness}
\gamma^{-1} \left \Vert \mathcal{A}-1 \right \Vert_{H^{s_1}\left( \mathbb{T}^{\nu +1} \right)}:= \delta \le \eta_\star.
\end{equation}
Consider a family of Lorentzian metrics, depending on $\omega \in \mathcal{O}_0$, of the form \eqref{eq:gmetric} where $A(t,x)=\mathcal{A}\left( \omega t ,x\right)$ is quasi-periodic with frequency vector $\omega \in \mathbb{R}^\nu$. 
Then, there exists a Lipschitz function $\alpha : \mathcal{O}_0 \rightarrow \mathbb{R}$ such that 
for any 
\begin{equation} \label{def:Oinfty1}
\omega \in \mathcal{O}^{2\gamma}_\infty := \left\{ \omega \in \mathcal{O}_0: | \omega \cdot \ell + \alpha^{-1}(\omega) \cdot j | > \frac{2\gamma}{\left< \ell, j \right>^\tau}, \,\,\forall (l, j ) \in \mathbb{Z}^{\nu+1}\setminus \{0 \} \right\},
\end{equation} 
there exists a global change of coordinates 
\begin{align*}
\mathcal{C}_\omega : \mathbb{R}_t \times \mathbb{S}^1_x  \rightarrow  \mathbb{R}_\tau \times \mathbb{S}^1_R , \quad
(t,x)\mapsto   (\tau, R),
\end{align*}
so that, in $(\tau,R)$ coordinates, the metric $g$ takes the form 
\begin{align}\label{eq:gmetrictauR}
g =  \Omega^2 \left( -\frac{1}{\alpha^2(\omega)} d\tau \otimes d\tau + dR \otimes dR \right), 
\end{align}
where $  \mathcal{O}_0 \ni \omega \mapsto \alpha(\omega) > 0$ is independent of $(\tau, R)$ but depends on $\omega$ in a Lipschitz manner, 
\begin{align*}
	\sup_{\omega \in \mathcal{O}_0}| \alpha(\omega)-1 | + \gamma \sup_{\omega_1\neq \omega_2} \frac{| \alpha(\omega_1)-\alpha(\omega_2)|}{|\omega_1- \omega_2|} \lesssim \delta \gamma.
\end{align*}
Moreover, the change of coordinates $\mathcal{C}_\omega$ maps quasi-periodic functions with respect to $t$ to quasi-periodic functions with respect to $\tau$ and verifies the following regularity and algebraic properties: 
\begin{itemize}[leftmargin=5.5mm]
\item (Parity preserving) For any function $f(t,x)$ that is even (respectively odd) in $x$, $(f \circ \mathcal{C}_\omega^{-1} )(\tau, R)$ is even (respectively odd) in $R$.
\item (Reversibility preserving) For any function $f(t,x)$ that is even (respectively odd) in $t$, $(f \circ \mathcal{C}_\omega^{-1} )(\tau, R)$ is even (respectively odd) in $\tau$.
\item The change of coordinates $\mathcal{C}_\omega: \mathbb{R}_t \times \mathbb{S}^1_x  \rightarrow  \mathbb{R}_\tau \times \mathbb{S}^1_R$ has the form 
\begin{eqnarray*} 
(t,x) \mapsto  \left(\tau(t,x), R(t,x)\right) :=\left( t + \frac{\alpha(\omega)}{2}\left( U(t,x)+ V(t,x) \right) , x + \frac{V(t,x)-U(t,x)}{2}\right), 
\end{eqnarray*}
where $U$ and $V$ are quasi-periodic functions, meaning that they can be written as $U := \mathcal{U}(\omega t, x)$, $V:= \mathcal{V} (\omega t, x)$, for some $\mathcal{U},\mathcal{V} : \mathbb{T}^{\nu+1} \rightarrow \mathbb{R}$, and verify the estimates
\begin{eqnarray} \label{es:UVth}
 \Vert  \mathcal{U} \Vert_{s}+
 \Vert  \mathcal{V} \Vert_{s}\lesssim_s \gamma^{-1} 
\Vert \mathcal{A}-1 \Vert_{s+2\iota+4}, \quad \forall s \ge s_0,\quad  s \in \mathbb{N}. 
\end{eqnarray}
\item (Tameness) For any quasi-periodic function $H:=\mathcal{H}(\omega t, x)$ with $\mathcal{H} \in H^s\left(\mathbb{T}^{\nu+1} \right)$, one has 
\begin{equation} \label{es:tameth}
\left \Vert H \circ \mathcal{C}_\omega^{-1} \right \Vert_{s} 
\lesssim  \Vert \mathcal{H} \Vert_s +    \gamma^{-1} \left \Vert  \mathcal{A}-1 \right \Vert_{s+\mu}\Vert \mathcal{H} \Vert_1, \quad \mu := s_0+ 2 \iota +4 ,\quad \forall s \ge s_0.
\end{equation}
\end{itemize}
Finally, the conformal factor $\Omega(\tau,R):=\mathfrak{O}(\omega \tau,R)$ in \eqref{eq:gmetrictauR} is quasi-periodic with respect to $\tau$, even in $\tau$ and $R$ and verifies the estimate  
\begin{eqnarray} \label{es:confor}
\Vert \mathfrak{O}^2 -1\Vert_s \lesssim  \Vert\mathcal{A}-1\Vert_{ s+\mu}, \quad \forall s \ge s_0.
\end{eqnarray}
\end{theorem}
\begin{proof}
	Theorem \ref{th:redumetric} follows from collecting the statements of Proposition \ref{prop:main1}, Proposition \ref{prop:main2}, Lemma \ref{lem:psidiff}, Lemma \ref{lem:psiest} and Lemma \ref{lem:confreg}. 
\end{proof}

\begin{remark}
As in \cite{reducimoser}, the Lipschitz regularity of $\alpha$ with respect to $\omega$ implies that the set $\mathcal{O}^{2\gamma}_\infty$ is asymptotically of full measure as $\gamma \rightarrow 0$, namely 
$$
\vert \mathcal{O}_0 \setminus  \mathcal{O}^{2\gamma}_\infty \vert \rightarrow 0, \quad \gamma \rightarrow 0. 
$$
\end{remark}

\begin{remark}Because of the extra factor $\alpha^{-2}(\omega)$ in front of the $d\tau^2$, the form of the new metric \eqref{eq:gmetrictauR} differs from that of the standard conformally flat metric \eqref{eq:locconflat}. This is due to the fact that we want to preserve the frequency vector $\omega$ in the definition of quasi-periodicity. Indeed, we can always go from \eqref{eq:gmetrictauR} to \eqref{eq:locconflat} by a rescaling $\tau \rightarrow \alpha^{-1} \tau$ but this then also changes $\omega$ to $ \alpha \omega$. 
\end{remark}
\begin{remark}\label{RemarkLipschitz}
In fact, we prove slightly stronger versions of essentially all the estimates above as we also control the Lipschitz dependence with respect to $\omega$ of all the aforementioned quantities. Compare for instance \eqref{es:UVth} with item \eqref{item:Ues} of Proposition \ref{prop:main1}. We can also handle the case where the function $\mathcal{A}$ may dependent on $\omega$ in a Lipschitz manner. 
\end{remark}

\begin{remark}
Theorem \ref{th:redumetric} can be viewed essentially as a statement concerning the reducibility of the $1+1$ quasi-periodic Lorentzian metric itself. The reducibility is only up to a conformal factor, cf~\eqref{eq:gmetrictauR}, since in general this is the best one can achieve (at least using only a change of coordinates). As a corollary, it reduces \eqref{eq:KG} to a wave or Klein-Gordon equation with submaximal pertubations, cf~Corollary \ref{cor:kgred} below. 
\end{remark}
\begin{remark} \label{rem:pcwpp}
If $\mathcal{A}$ is only even in $t$, but not in $x$, then the theorem still mostly holds (cf~Remark \ref{rem:tcwppr} below), apart from the conclusions concerning the parity preserving properties and the parity property in $x$ of the new coordinates. Similarly, if $\mathcal{A}$ is only even in $x$, but not in $t$, then the theorem still holds apart from the conclusions concerning the reversibility preserving properties and the parity property in $t$ of the new coordinates. The case where $\mathcal{A}$ is not even in $t$ and also not even in $x$ is treated separately just below, cf~Theorem \ref{th:redumetricws}. 
\end{remark}

\begin{remark} \label{rem:uni}
The form of the metric \eqref{eq:metricnull} is invariant under a change of coordinates of the form $(u,v) \rightarrow (f(u), g(v))$, typically called in language of General Relativity a \emph{residual gauge transformation}. However, from the proof, it will follow that the null coordinates that preserve the quasi-periodicity are \emph{uniquely defined}, up to constants. This means that the additional gauge freedom is fixed by the requirement of respecting the quasi-periodicity. 
\end{remark}
\noindent
A similar theorem can be obtained without the parity requirements on $\mathcal{A}$. In this case however, we are no longer able to preserve the original frequency vector and obtain the following variant. 

\begin{theorem}[Main theorem without parity] \label{th:redumetricws}
Let $\nu \in\mathbb{N}^\star$ and fix $s_0 > (\nu +7)/2$. Let $\mathcal{O}_0$ be a compact set of $\mathbb{R}^\nu$ of positive Lebesgue measure and let $\gamma \in (0,1)$ and $\iota := \nu+3$ be fixed diophantine parameters. 
Then, there exists $s_1 > s_0$ and $\eta_\star(s_1) >0$ such that the following holds. 
Let $\mathcal{A}: \mathbb{T}^{\nu+1} \rightarrow (0,+\infty)$ be a smooth function verifying the smallness assumption 
\begin{equation}\label{ineq:Asmallnesswp}
\gamma^{-1} \left|\left| \mathcal{A}-1\right|\right|_{H^{s_1}\left( \mathbb{T}^{\nu +1} \right)}:= \delta \le \eta.
\end{equation}
Consider a family of Lorentzian metrics, depending on $\omega \in \mathcal{O}_0$, of the form \eqref{eq:gmetric} where $A(t,x)=\mathcal{A}\left( \omega t ,x\right)$ is quasi-periodic with frequency vector $\omega \in \mathbb{R}^\nu$. 
Then, there exists two Lipschitz functions $\alpha_\pm : \mathcal{O}_0 \rightarrow \mathbb{R}$ such that 
for any $\omega \in \mathcal{O}_\infty:= \mathcal{O}^{2\gamma}_{\infty,+} \cap \mathcal{O}^{2\gamma}_{\infty,-}$, with 
\begin{equation} \label{eq:diophw}
\mathcal{O}^{2\gamma}_{\infty,\pm} := \left\{ \omega \in \mathcal{O}_0: | \omega \cdot \ell + \alpha_\pm^{-1}(\omega) \cdot j | > \frac{2\gamma}{\left< \ell, j \right>^\iota}, \,\,\forall (l, j ) \in \mathbb{Z}^{\nu+1}\setminus \{0 \} \right\},
\end{equation} the following holds. 
\begin{itemize}[leftmargin=5.5mm]
\item The functions $\alpha_\pm$ are close to $1$ in the Lipschitz norm
$$
\sup_{\omega \in \mathcal{O}_0}| \alpha_\pm(\omega)-1 | + \gamma \sup_{\omega_1\neq \omega_2} \frac{| \alpha_\pm(\omega_1)-\alpha_\pm(\omega_2)|}{|\omega_1- \omega_2|} \lesssim \delta \gamma.
$$
\item For any $\omega \in \mathcal{O}_\infty$, there exists a global change of coordinates 
\begin{align*}
	\mathcal{C}_\omega : \mathbb{R}_t \times \mathbb{S}^1_x \rightarrow \mathbb{R}_\tau \times \mathbb{S}^1_R, \quad (t,x) \mapsto (\tau, R):=\mathcal{C}_\omega(t,x)
\end{align*} 
so that, in $(\tau,R)$ coordinates, the metric $g$ takes the form 
\begin{eqnarray}\label{eq:gmetrictauRwsth}
g:=   \Omega^2 \left( - \alpha_- \alpha_+ d\tau^2 + dR^2 + 2 \left(\alpha_- -\alpha_+ \right)\, d\tau dR \right),  
\end{eqnarray}
\item The change of coordinates $\mathcal{C}_\omega$ maps $\omega$-quasi-periodic functions with respect to $t$ to quasi-periodic functions with respect to $\tau$ and verifies tame estimates as in \eqref{es:tameth}.
\item Finally, the conformal factor $\Omega(\tau,R)=\mathfrak{O}(\omega \tau,R)$ in \eqref{eq:gmetrictauRwsth} is quasi-periodic with respect to $\tau$ and verifies the estimate \eqref{es:confor}.
\end{itemize}
\end{theorem}

\subsection{Applications to the Klein-Gordon equation with maximal order pertubations}

As explained above, the change of coordinates provided by Theorems \ref{th:redumetric} and \ref{th:redumetricws}  can be used to transform the Klein-Gordon equation \eqref{eq:KG} to an equation with submaximal perturbations at most. This is stated to the following corollary.

\begin{cor} \label{cor:kgred}
Consider the Klein-Gordon equation \eqref{eq:KG} with quasi-periodic coefficients as in \eqref{def:qpc12} verifying the following algebraic properties:
\begin{itemize}[leftmargin=5.5mm]
\item[a.] $\mathcal{B}^{xx}$ and $\mathcal{B}$ are even in $\varphi$ and $x$ independently, 
\item[b.] $\mathcal{B}^x$ is even in $\varphi$ and odd in $x$, 
\item[c.] $\mathcal{B}^t$ is odd in $\varphi$ and even in $x$. 
\end{itemize}
With $A^2= 1- B^{xx}$, assume that the conditions of Theorem \ref{th:redumetric} hold, let $\mathcal{C}^{-1}_\omega=(\tau, R)$ be the corresponding change of coordinates and $\mathcal{O}^{2\gamma}_\infty$ be as in \eqref{def:Oinfty1}. Then, if $\psi$ solves \eqref{eq:KG} with $\omega \in \mathcal{O}^{2\gamma}_\infty$, $\phi:= \psi \circ \mathcal{C} ^{-1}_\omega$ verifies the Klein-Gordon equation 
\begin{align} \label{eq:wtauR2prop}
\left(\square_g - \mathtt{m} \right) \phi &:=\left(- \alpha^2(\omega) \partial_{\tau}^2  + \partial_{R}^2  - \mathtt{m}\right) \phi  
 = G^R \partial_{R} \phi  + G^\tau \partial_{\tau} \phi  +G \phi, 
\end{align}
where 
\begin{align*}
	G^R=\mathcal{G}^R(\omega \tau,R), \quad  G^\tau=\mathcal{G}^\tau(\omega \tau,R), \quad G=\mathcal{G}(\omega \tau,R),
\end{align*}
for some functions $\mathcal{G}^R$, $\mathcal{G}^\tau$, $\mathcal{G}: \mathbb{T}^\nu_{\varphi'} \times \mathbb{T}_R \rightarrow \mathbb{R}$ verifying the following algebraic properties:
\begin{itemize}[leftmargin=5.5mm]
\item[a.] $\mathcal{G}$ is even\footnote{The variable $\varphi' \in \mathbb{T}^\nu$ is the replacement of $\varphi$ corresponding to the change of coordinates, cf~\eqref{def:Psi}. To ease the writing, we will sometimes drop the $'$ when we work purely in the new coordinate system.} in $\varphi'$ and $R$ independently, \label{item:G}
\item[b.] $\mathcal{G}^R$ is even in $\varphi'$ and odd in $R$,  \label{item:GR}
\item[c.] $\mathcal{G}^\tau$ is odd in $\varphi'$ and even in $R$.  \label{item:GT}
\end{itemize}
as well the estimates
\begin{eqnarray} \label{es:mathcalg}
\|\mathcal{G}^R  \|_s+
\| \mathcal{G}^\tau \|_s+
\| \mathcal{G} \|_s \lesssim \max_{a \in \left\{  \mathcal{B}^{xx}, \mathcal{B}^{x}, \mathcal{B}^{t}, \mathcal{B} \right\}} \gamma^{-1}  \|    a    \|_{H^{s+s_0+2\iota+5}(\mathbb{T}^{\nu+1})}.
\end{eqnarray}
\end{cor} 
\noindent
The fact that Theorem \ref{th:redumetricws} is a statement about a pure change of coordinates, instead of a more complex transformation of phase space, has also benefit for the treatment of the lower order terms. We explain in this paper how to treat in a relatively simple manner the terms of order $1$ in \eqref{eq:wtauR2prop}, leaving only terms of order $0$ with small coefficients, which have already been treated extensively in the literature.  \\ \\
First, one can get rid of the terms with time derivatives\footnote{Equivalently, one can remove the spatial derivative terms. One cannot however remove together the spatial and time derivative terms using a pure multiplication operator, unless those terms are proportional to a standard \emph{null form} $Q_0(B, \psi)= \partial_t B \partial_t \psi- \partial_R B \partial_R \psi$.}, just by a multiplication operator. 
\begin{prop} \label{prop:redutime1}
Under the assumptions of Corollary \ref{cor:kgred}, consider the operators
\begin{align*}
\mathcal{L}_1: &= - \alpha^2(\omega) \partial_{\tau}^2 + \partial_{R}^2 - \mathtt{m}-  G^R \partial_R - G^\tau \partial_\tau, \\
\mathcal{L}_2: &= - \alpha^2(\omega) \partial_{\tau}^2 + \partial_{R}^2 - \mathtt{m}-  G^R \partial_R, 
\end{align*}
which differ only by the first order term $- G^\tau \partial_\tau$.  There exists a quasi-periodic function $P(\tau,R):=\mathcal{P}(\omega t , R)$, which is even in $\tau$ and in $R$, such that the operator 
$$\widetilde{\mathcal{L}_2}:= \mathcal{L}_2+ 2 P^{-1} \left( \partial_R P \right) \partial_R$$verifies, for any regular function $h:=h(\tau, R)$ defined on $\mathbb{R} \times \mathbb{S}^1$, 
$$\widetilde{\mathcal{L}_2} \left( P h \right) = P \mathcal{L}_1 h + G_2 h,$$
where 
\begin{itemize}[leftmargin=5.5mm]
\item $G_2$ is a quasi-periodic function $G_2(\tau, R):= \mathcal{G}_2(\omega \tau, R)$, 
\item $\mathcal{G}_2$ and $\mathcal{P}$ are even in both $\varphi$ and $R$, 
\item $\mathcal{G}_2$ and $\mathcal{P}$ verifies the estimates
\begin{eqnarray} \label{es:mathcalgp}
\Vert \mathcal{G}_2  \Vert_s+
\Vert  \mathcal{P} \Vert_s \lesssim \max_{a \in \left\{  \mathcal{B}^{xx}, \mathcal{B}^{x}, \mathcal{B}^{t}, \mathcal{B} \right\}} \gamma^{-1}  \|    a    \|_{H^{s+s_0+3\iota+7}(\mathbb{T}^{\nu+1})}.
\end{eqnarray}
\end{itemize}
\end{prop}
\noindent
The above proposition means that up to changing the first order term in $\partial_R$ and the $0$th-order term, we can get rid of the time derivative first order term. Relabelling $G^R$ and $G$, 
we can thus focus on the operator $\mathcal{L}_2:= - \alpha^2(\omega) \partial_{\tau}^2  + \partial_{R}^2  - \mathtt{m}-  G^R \partial_R$ and Klein-Gordon equations of the form 
\begin{eqnarray}
 \label{eq:wtauRP2}
(\mathcal{L}_2 + G) \phi =0, 
\end{eqnarray}
where $\omega \in \mathcal{O}^{2\gamma, \infty}$, $G(\tau,R):=\mathcal{G}(\omega \tau, R)$ and $G^R(\tau, R):= \mathcal{G}^R(\omega \tau, R)$ are quasi-periodic functions verifying the algebraic properties of Corollary \ref{cor:kgred} as well as the estimates 
\begin{eqnarray} \label{es:mathcalggr}
\Vert \mathcal{G}  \Vert_s+
\Vert  \mathcal{G}^R \Vert_s \lesssim \max_{a \in \left\{  \mathcal{B}^{xx}, \mathcal{B}^{x}, \mathcal{B}^{t}, \mathcal{B} \right\}}  \gamma^{-1} \|    a    \|_{H^{s+s_0+3\iota+8}(\mathbb{T}^{\nu+1})}.
\end{eqnarray}
\noindent 
In order to remove the remaining first order derivative term, we need to apply a phase space transformation, which will therefore act on both $\phi$ and $\phi_t$.  To this end, it is useful to introduce another parametrization of the phase space.  Consider the first order operators
 \begin{align*}
 K =-  i \alpha \partial_\tau + D_{\mathrm{m}} , \quad 
 \underline{K} = i \alpha \partial_\tau + D_{\mathrm{m}}, 
 \end{align*}
where $D_{\mathrm{m}}= \mathrm{Op}( \sqrt{|\xi|^2+\mathrm{m}})$ is a Fourier multiplier and $\alpha= \alpha(\omega)$ is the coefficient appearing in the second time derivative term of $\mathcal{L}_2$. We then introduce the complex variables\footnote{In order to ease the comparison, we follow the notations and normalization of \cite{berti2024reducibility}.}
\begin{align*}
	\left[ \begin{array}{c} u \\ \underline{u} \end{array} \right] =   \mathcal{C} \left[ \begin{array}{c} \phi \\ v \end{array} \right], \quad \mathcal{C}: =  \frac{1}{\sqrt{2}} \left( \begin{array}{cc} D_m & - \alpha i \\ D_m &\alpha i \end{array} \right), \quad v: = \partial_t \phi.
\end{align*}
and note that
\begin{align*}
\left[ \begin{array}{c} u \\ \underline{u} \end{array} \right] &=    \frac{1}{\sqrt{2}} \left[ \begin{array}{c} D_m\phi - i \alpha v\\ D_m \phi + i \alpha v \end{array} \right]=  \frac{1}{\sqrt{2}} \left[ \begin{array}{c} K(\phi)\\ \underline{K} (\phi)\end{array} \right], \quad \mathcal{C}^{-1}:  =  \frac{1}{\sqrt{2}} \left( \begin{array}{cc} D_m^{-1}  & D_m^{-1} \\  \frac{i}{\alpha}    &-\frac{i}{\alpha } \end{array} \right).
\end{align*} 
The wave equation \eqref{eq:wtauRP2} can then be rewritten as 
\begin{eqnarray*}
- K \underline{K} \phi &=& G^R \partial_{R} \phi  + G \phi,
\end{eqnarray*}
or equivalently, in a first order schematic form, with $U := \left(\begin{array}{c} u \\ \underline{u} \end{array}\right)$, as  
\begin{eqnarray} \label{eq:Ufo}
\partial_t U := \mathcal{D}_1 U+  \mathcal{D}_0 U + \mathcal{D}_{-1} U,
\end{eqnarray}
where $\mathcal{D}_0$ (respectively $\mathcal{D}_{-1}$) are matrices of pseudo-differential operators depending quasi-periodically on $\tau$ of order $0$ (respectively of order $-1$) and 
$$
\mathcal{D}_1:=  \frac{i}{\alpha} \left( \begin{array}{cc} D_\mathrm{m} & 0 \\ 0 & -D_\mathrm{m} \end{array} \right),
$$
is a matrix of pseudo-differential operators of order $1$. \\ \\
We can then state our result as follows.
\begin{prop} \label{prop:reducall1} There exists $\eta_2 >0$ such that if  $G^R:=\mathcal{G}^R(\omega t, x)$ verifies the smallness assumption  
\begin{eqnarray} \label{ass:smallnessGR}
 \Vert \mathcal{G}^R \Vert_{s_0}  \le \eta_2, 
\end{eqnarray}
then, there exists a bounded, quasi-periodic in $\tau$, transformation $\mathcal{T}(\omega t): H^s( \mathbb{T} ; \mathbb{C}^2) \rightarrow  H^s( \mathbb{T} ; \mathbb{C}^2)$, $s \ge s_0$, such that, 
 \begin{itemize}[leftmargin=5.5mm]
\item  $U$ solves \eqref{eq:Ufo} if and only if $\mathcal{Z}:= \mathcal{T} U$ solves 
\begin{eqnarray} \label{eq:nofur}
\partial_t Z := \mathcal{D}_1 Z + \widetilde{\mathcal{D}}_{-1} Z   , 
\end{eqnarray}
where  
$\widetilde{\mathcal{D}}_{-1}$ is a pseudo-differential operator of order $-1$. 
\item For any $U=(u,\underline{u}) \in H^s( \mathbb{T} ; \mathbb{C}^2)$, $s \ge s_0$, and any $\omega \in \mathcal{O}^{2\gamma}_\infty$, 
\begin{align*}
\Vert\widetilde{\mathcal{D}}_{-1} (u,\underline{u})  \Vert_{s} &\lesssim  \left(\Vert \mathcal{G}^R \Vert _{s +2\iota }+
\Vert \mathcal{G} \Vert _{s} \right)\left(
	1 + \Vert \mathcal{G}^R \Vert _{s_0 +2\iota } 
	\right)\Vert (u,\underline{u})  \Vert _{s_0-1} \\
	& \quad  +\left(\Vert \mathcal{G}^R \Vert _{s +2\iota }+
\Vert \mathcal{G} \Vert _{s} \right) \left(1+\Vert \mathcal{G}^R \Vert _{s_0+2\iota } \right)\Vert (u,\underline{u})  \Vert _{s-1} , \\
\Vert (\mathcal{T}-\mathrm{Id}) (u,\underline{u}) \Vert _s &\lesssim  \Vert \mathcal{G}^R \Vert _{s +2\iota }\left(
	1 + \Vert \mathcal{G}^R \Vert _{s_0 +2\iota } 
	\right)\Vert (u,\underline{u})  \Vert _{s_0}
	+\Vert \mathcal{G}^R \Vert _{s_0 +2\iota } \left(1+\Vert \mathcal{G}^R \Vert _{s_0+2\iota } \right)\Vert (u,\underline{u})  \Vert _{s} . 
\end{align*}
\item $\mathcal{T}$ and $\widetilde{\mathcal{D}}_{-1}$ are real-to-real, reversibility and parity-preserving operators\footnote{The definitions of these algebraic properties for operators are recalled in Section \ref{se:apop}.}.
\end{itemize}
\end{prop}

\begin{remark}
Similarly to Remark \ref{RemarkLipschitz}, we in fact prove slightly stronger versions of essentially all the estimates above as we also control the Lipschitz dependence with respect to $\omega$ of all the aforementioned quantities.  
\end{remark}

\begin{remark}
	The statement of the above proposition is close to that Proposition 9.5 of \cite{berti2024reducibility} with, in our opinion, the main benefits lying in the slightly simplified proof, cf~Section \ref{se:proopr:reducall1} below.  We do not pursue here the full reduction of the terms of order $-1$, since this has already been treated somewhat extensively in the literature.
\end{remark}

\noindent 
In the case where $A$ does not verify the parity requirements, one could try to develop a similar reduction of the lower order terms, but for simplicity in the exposition, we state here only the following rather direct corollary of Theorem \ref{th:redumetricws}. 

\begin{cor} \label{cor:cwp} Under the assumptions of Theorem \ref{th:redumetricws}, for any initial data (f,g) in $H^{s+1}(\mathbb{S}^1) \times H^s(\mathbb{S}^1)$, $s \ge s_0$, the unique global in time solution to the geometric wave equation 
\begin{align*}
\begin{dcases}
	\square_g \psi=0, \\
\psi(t=0) = f, \\
 \partial_t \psi(t=0) =g, 
\end{dcases}
\end{align*}
where $\square_g$ is the wave operator introduced in \eqref{eq:waveop}, is almost-periodic\footnote{We recall that almost-periodicity in time means that the solution can be represented as a convergent series of oscillating in time exponentials of the form $\sum_{j \in \mathbb{N}} e^{i\lambda_j t} \psi_j(x)$, $\lambda_j \in \mathbb{R}$. } in time and verify the estimate
$$
\sup_{t \in \mathbb{R}} \left( \Vert \psi(t, \cdot) \Vert_{s+1} + \Vert \partial_t \psi(t, \cdot) \Vert_{s}\right)  \lesssim \Vert f \Vert_{s+1} + \Vert g \Vert_{s}.
$$

\end{cor}

\subsection{Comments on time-periodic pertubations of AdS}
While linear Klein-Gordon or wave equations appear naturally in many situations arising from physics, our main motivation for the study of \eqref{eq:KG} comes from the study of perturbations of the so-called \emph{Anti-de-Sitter space} (AdS), especially in the case of spherically symmetric models as in \cite{PhysRevLett.107.031102, PhysRevLett.111.051102}. In that case, the equations take the form of a 1d, geometric, quasi-linear wave equation. After linearization on a non-trivial solution, one is left with a Klein-Gordon type equation with non-constant coefficients.  \\ \\
Numerical and perturbative (i.e.~as formal power series) time-periodic solutions of the Einstein-Klein-Gordon equation in spherical symmetry have been constructed in \cite{PhysRevLett.111.051102}. One expects that proving a reducibility result for the wave operator associated to time-periodic or quasi-periodic perturbations of the AdS metric (within spherical symmetry) will be one of the main ingredients of an actual proof of existence of these solutions. \\ \\
The equation \eqref{eq:KG} encapsulates some of the difficulties that one expects to face in the AdS case: the non-constant coefficients, especially for the second order derivative term (arising from the quasi-linearity) and the discrete spectrum, which is due to the reflective boundary conditions in the AdS case. Equation \eqref{eq:KG} does however exhibit some significant simplifications, in particular, the eigenfunctions associated to the non-perturbed linear problem, \eqref{eq:KG} with $B^{xx}=B^x=B^t=B$, are just the standard trigonometric functions and the boundary conditions in AdS have been replaced by a periodic spatial domain $\mathbb{S}^1_x$.  \\ \\
We note however that the problem of constructing time-periodic solutions to \emph{semi-linear} wave equations on a fix Anti-de-Sitter space has been recently addressed in \cite{SakisJacques, chatzikaleas2023time} and independently in \cite{berti2023time}. In particular, the works \cite{SakisJacques, chatzikaleas2023time} relied on detailed Fourier analysis using the eigenfunctions associated with Klein-Gordon type operators on AdS. 


\subsection{Overview of the paper}
In Section \ref{se:preli}, we present both the (classical) functional framework as well as several analytical tools which we will need, such as the statement concerning the reducibility of vector fields (taken from \cite{reducimoser}) that we will use as a black box for the construction of the null coordinates. Section \ref{se:ncqlm} contains the proof of Theorem \ref{th:redumetric} and thus the main novelties of this paper. The case without parity is discussed at the end of this section, including the proof of Corollary \ref{cor:cwp} on Sobolev stability for the geometric wave equation. Finally, Section \ref{se:rkg} contains the analysis of the Klein-Gordon equation in the novel system of coordinates. The section starts with the estimates for the new lower order terms in Section \ref{se:kgtR} and then provides the proof of propositions \ref{prop:redutime1} and \ref{prop:reducall1} in Section \ref{se:redfirsto}. 

\subsection{Acknowledgements}
We would like to thank Massimiliano Berti who suggested to us the investigation of the case without parity during the visit of J.S.~to SISSA.

\section{Preliminaries} \label{se:preli}
\subsection{Basic notations and function spaces}
We fix initial parameter $s_0 > (\nu+1)/2 +3$. The diophantine parameters will be denoted by $\gamma \in (0, 1)$ and $\iota := \nu+3$. With a few exceptions (such as $\iota$ instead of $\tau$), our notations follows that of \cite{berti2024reducibility} or \cite{reducimoser}.  \\ \\
We denote by $H^s$ the Sobolev spaces defined by
\begin{align*}
H^s: &= H^s(\mathbb{T}^{\nu+1}, \mathbb{C} ) 
:= \left\{ f(\varphi, x) \in L^2 (\mathbb{T}^{\nu+1}, \mathbb{C} )\,\,:\,\, \Vert f \Vert_s^2:= \sum_{\ell \in \mathbb{Z}^\nu, \,j \in \mathbb{Z}} \left< \ell, j \right>^{2s} | f_{\ell, j}|^2 < \infty \right\}, 
\end{align*}
where $\left< \ell, j \right>:= \max\{ 1, | \ell|, |j|\}$ and $f_{\ell, j}$ are the Fourier coefficients of $f$ that are given by
\begin{align*}
	f_{\ell, j} := \frac{1}{(2\pi)^{\nu+1}} \int_{\mathbb{T}^{\nu+1}} f(\varphi, x) e^{-\left(\ell. \varphi + jx\right) } d\varphi dx. 
\end{align*}
We also consider functions depending on an additional parameter $\omega \in \mathcal{O}$, where  $\mathcal{O}$ is a compact subset of $\mathbb{R}^\nu$. For a $1$-parameter family of $H^s$-functions $f : \mathcal{O} \rightarrow H^s$, we consider the Lipschitz norm 
\begin{equation}\label{def:lipnom}
\Vert f \Vert_{s}^{\gamma, \mathcal{O}} := \sup_{\omega \in \mathcal{O}} \Vert f (\omega) \Vert_s + \gamma \sup_{\substack{\omega_1 \neq \omega_2,\\ \omega_1,\, \omega_2 \in \mathcal{O}}} \frac{\Vert f (\omega_1)-f(\omega_2 \Vert_{s-1}}{|\omega_1-\omega_2|}.
\end{equation}
For a function $\alpha(\omega)$ depending only on $\omega$ (thus independent of $(\varphi,x)$), its Lipschitz norm is given by 
\begin{align*}
	| \alpha |^{\gamma, \mathcal{O}} :=  | \alpha |^{\gamma}:=  \sup_{\omega \in \mathcal{O}} | \alpha (\omega) | + \gamma \sup_{\substack{\omega_1 \neq \omega_2,\\ \omega_1,\, \omega_2 \in \mathcal{O}}} \frac{| \alpha (\omega_1)-\alpha(\omega_2) |}{|\omega_1-\omega_2|}.
\end{align*} 
We will also rely on the standard spaces $L^\infty\left(\mathbb{T}^d, \mathbb{C}\right)$ and $W^{s, \infty}\left(\mathbb{T}^d, \mathbb{C}\right)$, $d \ge 1$, with their respective norms
\begin{eqnarray*}
\Vert u \Vert_{L^\infty} = \sup_{ x \in \mathbb{T}^d } |u(x)|, \quad \Vert u \Vert_{W^{s,\infty}} = \sum_{|s'| \le s} | \partial_x^{s'} u |_{L^\infty},
\end{eqnarray*}
where in the last formula, $s'$ denote a multi-index, $|s'|$ its length and thus $\partial_x^{s'}$ is a differential operator of order $|s'|$. \\ \\
As above, for a function $u_\omega$ depending in a Lipschitz way on an external parameter $\omega \in \mathcal{O}$, we shall also use the norms
$\Vert u \Vert^{\gamma, \mathcal{O}}_{L^\infty}$, $\Vert u \Vert^{\gamma, \mathcal{O}}_{W^{s,\infty}}$, which are defined similarly to \eqref{def:lipnom}. \\ \\
We recall the following standard functional inequalities (cf~(2.5) and (2.6) in \cite{berti2024reducibility}).  For $s \ge s_0$ and $f,g \in H^s(\mathbb{T}^{\nu+1})$, we have that
\begin{enumerate}[leftmargin=5.5mm]
\item Sobolev embedding:
\begin{eqnarray} \Vert f \Vert_{L^\infty} \lesssim \Vert f \Vert_{s_0}, \label{ineq:sob}
 \end{eqnarray}
\item Tame algebra estimate:
\begin{eqnarray} 
 \label{es:tamealgebra}
\Vert f g \Vert_{s}^{\gamma, \mathcal{O}} \lesssim \Vert f  \Vert_{s}^{\gamma, \mathcal{O}} \Vert g \Vert_{s_0}^{\gamma, \mathcal{O}} + \Vert f \Vert_{s_0}^{\gamma, \mathcal{O}} \Vert g \Vert_{s}^{\gamma, \mathcal{O}}.
\end{eqnarray}
\item Interpolation estimate: for any $a_0, b_0 \ge 0$, $p, q > 0$, $u \in H^{a_0+p+q}$, $v\in H^{b_0+p+q}$, 
\begin{eqnarray}\label{ineq:interp}
\Vert u \Vert^{\gamma, \mathcal{O}}_{b_0+p} \Vert v \Vert^{\gamma, \mathcal{O}}_{a_0+p} \le \Vert u \Vert^{\gamma, \mathcal{O}}_{a_0+p+q} \Vert v \Vert^{\gamma, \mathcal{O}}_{b_0}+ \Vert u \Vert^{\gamma, \mathcal{O}}_{a_0} \Vert v \Vert^{\gamma, \mathcal{O}}_{b_0+p+q}.
\end{eqnarray} 

\end{enumerate}


\subsection{Reducibility of vector fields}
We recall in this section the results of \cite{reducimoser} concerning the reducibility of vector fields that will be used later in this article\footnote{The original results obtained in \cite{reducimoser} are more general than those needed here. For instance, the variable $x$ can be vectorial, while here we only consider $x \in \mathbb{S}^1$. We have adapted the statement above to fit the purpose of the present article.}. 

\begin{prop}[Proposition 3.4 in \cite{reducimoser}] \label{prop:redvf} Let $\gamma \in (0,1)$ and $\iota\ge \nu+3$. 
Consider a Lipschitz family of vector fields on $\mathbb{T}^{\nu+1}$,
\begin{align*}
\omega \mapsto	X_0(\omega) := \omega \cdot \frac{\partial}{\partial \varphi} + \left( 1+ a_0(\varphi, x; \omega) \right) \frac{\partial}{\partial x},
\end{align*}
depending on a parameter $\omega \in \mathcal{O}_0$, where $\mathcal{O}_0 \subset \mathbb{R}^\nu$ is a compact set of positive Lebesgue measure and where $ \omega \rightarrow a_0(\,\cdot \,, \,\cdot\,\,; \omega) \in H^s( \mathbb{T}^{\nu +1}), \,\,\forall s \ge s_0$. Then, there exists $s_1 \ge s_0 + 2\iota +4$ and $\eta_\star=\eta_\star(s_1) >0$ such that if 
\begin{eqnarray} \label{prop:reduassum}
\gamma^{-1} \Vert a_0 \Vert_{s_1}^{\gamma, \mathcal{O}_0} := \delta \le \eta_\star, 
\end{eqnarray}
then there exists a Lipschitz function $m : \mathcal{O}_0 \rightarrow \mathbb{R}$ such that 
\begin{equation*}
\vert m_\infty(\omega)-1\vert^\gamma  \le \gamma \delta,
\end{equation*}
and in the set 
\begin{equation}\label{def:Oinfty}
\mathcal{O}^{2\gamma}_\infty := \left\{ \omega \in \mathcal{O}_0: | \omega \cdot\ell + m_\infty(\omega) \cdot j | > \frac{2\gamma}{\left< \ell, j \right>^\iota}, \,\,\forall (l, j ) \in \mathbb{Z}^{\nu+1}\setminus \{0 \} \right\},
\end{equation} the following holds. There exists a map 
\begin{align*}
	\beta : \mathbb{T}^{\nu+1} \times \mathcal{O}^{2\gamma}_\infty \rightarrow \mathbb{R}, \quad \Vert \beta \Vert_{s}^{\gamma, \mathcal{O}_\infty^{2\gamma} } \lesssim_s \gamma^{-1}
\Vert a_0 \Vert_{s+2\iota+4}^{\gamma, \mathcal{O}_0}, \quad \forall s \ge s_0,\quad  s \in \mathbb{N},  
\end{align*}
so that $\Psi: (\varphi, x ) \rightarrow (\varphi, x + \beta(\varphi, x; \omega) = (\varphi, \hat x)$ is a diffeomoprhism of $\mathbb{T}^{\nu+1}$ and, $\forall \omega \in \mathcal{O}^{2\gamma}_\infty$, 
\begin{align*}
\Psi_\star X_0 :&= \omega \cdot \frac{\partial}{\partial \varphi} + \left(\omega \cdot \partial_{\varphi} \beta+ \left(1+ a_0\right) \left(1+\partial_x \beta\right) \right) \circ \Psi^{-1} \cdot \frac{\partial}{\partial \hat x}, \\
&= \omega \cdot \frac{\partial}{\partial \varphi} + m_\infty(\omega) \frac{\partial}{\partial \hat x}.
\end{align*}
\end{prop}

\begin{remark} \label{rem:uniq}
We thus see that the function $\beta$ that appears in the definition of $\Psi$ is solution to the inhomogeneous equation
$$
\omega \cdot \partial_{\varphi} \beta +\left( 1+ a_0 \right) \partial_x \beta = m_\infty(\omega)-(1+a_0).
$$
Up to constants, it is the only solution to this equation, since the kernel of the transport operator $\omega. \partial_{\varphi}  +\left( 1+ a_0 \right) \partial_x$ is mapped by $\Psi$ to the kernel of $\omega . \frac{\partial}{\partial \varphi} + m_\infty(\omega) \frac{\partial}{\partial \hat x}$ which contains only the constants, in view of the diophantine properties of $\omega$ and $m_\infty$, cf~\eqref{def:Oinfty}. This fact is responsible for the uniqueness, up to constants, of the coordinates provided by Theorem \ref{th:redumetricws}, cf~Remark \ref{rem:uni}.
\end{remark}

\subsection{Diffeomorpshims of the torus close to the identity}

We also recall the following lemma. 

\begin{lemma}[Lemma A.3 in \cite{reducimoser}]\label{lem:difftor}
Let $d \in \mathbb{N}$ and consider $p \in W^{s,+\infty}(\mathbb{T}^d, \mathbb{R}^d)$, $s\ge 1$, with $\Vert  p \Vert _{W^{1,\infty}} \le 1/2$. Let $f(x)= x +p(x)$. Then, we have that
\begin{enumerate}[leftmargin=5.5mm]
\item $f: \mathbb{T}^d \rightarrow \mathbb{T}^d$ is a diffeomorphism, its inverse has the form $f^{-1}(y)=y +q(y)$ with $q \in W^{s, \infty}$ verifying
\begin{align*}
	\Vert  q \Vert _{W^{s, \infty}} \lesssim_{d,s} \Vert  p \Vert _{W^{s,\infty}}.
\end{align*}
\item \label{es:difftor} For any function $h \in H^s(\mathbb{T}^{d})$, one has 
\begin{align*}
	\Vert  h \circ f \Vert _s &\leq \Vert h\Vert _s + C \left( \Vert  p \Vert _{W^{1,\infty}}   \Vert  h \Vert _s+ \Vert  \partial p \Vert _{W^{s-1, \infty}}\Vert  h \Vert _1 \right) \\
\Vert  h \circ f-h \Vert _s &\leq C \left( \Vert  p \Vert _{L^\infty} \Vert  h \Vert _{s+1} + \Vert  p \Vert _{W^{s,\infty}}\Vert  h \Vert _2 \right).
\end{align*}  
\item If $ \omega\in \mathcal{O} \subset  \mathbb{R}^\nu \rightarrow p_\omega$ is Lipschitz with respect to $\omega$ and verify, uniformly in $\omega$,  $\Vert  p_\omega \Vert _{W^{1,\infty}} \le \frac{1}{2}$, then $q:=q_\omega$ is Lipschitz with respect to $\omega$ and 
\begin{align*}
	\Vert  q \Vert ^{\gamma, \mathcal{O}}_{W^{s,\infty}} &\lesssim  \Vert  p \Vert ^{\gamma, \mathcal{O}}_{W^{s,\infty}}, \\  
\Vert u \circ f \Vert^{\gamma, \mathcal{O}}_s &\leq   \Vert u \Vert^{\gamma, \mathcal{O}}_{s} + C \left(\Vert u \Vert^{\gamma, \mathcal{O}}_{s} \Vert  p \Vert_{W^{1,\infty}}^{\gamma, \mathcal{O}}+ \Vert u \Vert^{\gamma, \mathcal{O}}_{2} \Vert \, p\, \Vert_{W^{s,\infty}}^{\gamma, \mathcal{O}}\right).
\end{align*}  
\end{enumerate}
\end{lemma}
\subsection{Operators: conventions, norms}
This subsection is not needed for the proof of Theorem \ref{th:redumetricws}, but is required for the proof of Proposition \ref{prop:reducall1}. 

\subsubsection{Algebraic properties of operators} \label{se:apop}
We follow here the conventions of \cite{berti2024reducibility}, apart from the change of letter for the real subspace $\mathcal{R}$. The definitions below are actually used mostly in Section \ref{se:rkg} and in particular are not needed for the proof of Theorem \ref{th:redumetric}. 
Let 
\begin{align*}
\mathcal{R}:&=  \left\{ (u^+, u^-) \in L^2(\mathbb{T}^{\nu+1}, \mathbb{C}^2) : \, \overline{u^-}=u_+\right\},  \\
\bold X :&=  (X \times X) \cap \mathcal{R}, \quad X:= \left\{ u \in L^2(\mathbb{T}^{\nu+1}, \mathbb{C})\,: u(\varphi, x) = - \overline{u(-\varphi, x)} \right\}, \\
\bold Y :&=  (Y \times Y) \cap \mathcal{R}, \quad Y:= \left\{ u \in L^2(\mathbb{T}^{\nu+1}, \mathbb{C})\,: u(\varphi, x) =  \overline{u(-\varphi, x)} \right\},  \\
O :&=  \left\{ u \in L^2(\mathbb{T}^{\nu+1}, \mathbb{C})\,: u(\varphi, x) =  - u(\varphi, -x) \right\},  \\
P : &= \left\{ u \in L^2(\mathbb{T}^{\nu+1}, \mathbb{C})\,: u(\varphi, x) =  u(\varphi, -x) \right\}.
\end{align*}
\begin{definition} \label{def:parityoper}
We give the following definitions.
\begin{enumerate}[leftmargin=5.5mm]
\item An operator $Q: H^s(\mathbb{T}^{\nu+1}; \mathbb{C}) \rightarrow H^{s'}(\mathbb{T}^{\nu+1}; \mathbb{C})$ is said to be 
\begin{itemize}[leftmargin=5.5mm]
\item reversible if $Q: X \rightarrow Y$ and $Q:Y \rightarrow X$, 
\item reversibility preserving if $Q: X \rightarrow X$ and $Q: Y \rightarrow Y$,  
\item parity preserving if $Q: O \rightarrow O \times O$ and $Q: P \rightarrow P \times P$. 
\end{itemize}
\item A matrix of linear operators $A:= \left(\begin{array}{cc} A^+_+& A^-_+ \\ A^+_-& A^-_- \end{array}\right)$, where each $A^\sigma_{\sigma'}: H^s(\mathbb{T}^{\nu+1}) \rightarrow H^{s'}(\mathbb{T}^{\nu+1})$, $\sigma, \sigma' = \pm$ is an operator, is called 
\begin{itemize}[leftmargin=5.5mm]
\item real-to-real if $A: \mathcal{R} \rightarrow \mathcal{R}$. 
\item reversible if $A$ is real-to-real and $A:\bold X \rightarrow \bold Y$ and $A:\bold Y \rightarrow \bold X$, 
\item reversibility preserving if $A: \bold X \rightarrow \bold X$ and $A: \bold Y \rightarrow \bold Y$,  
\item parity preserving if each component of $A$, $A^\sigma_{\sigma'}$ is parity preserving. 
\end{itemize}
\end{enumerate}
 \end{definition}

\subsubsection{Pseudo-differential operators conventions}This subsection is not needed for the proof of Theorem \ref{th:redumetricws}, but is required to follow the proof of Proposition \ref{prop:reducall1}. \\ \\
We follow the definitions and conventions of pseudo-differentials acting on periodic functions taken from \cite{berti2024reducibility}. 
For any $m\in \mathbb{R}$, we denote by $\Gamma^m$ the set of symbols of order $m$, i.e.~the set of $C^\infty(\mathbb T \times \mathbb{R} ; \mathbb{C})$ functions $a:= a(x,\xi)$ such that for all $\alpha, \beta \in \mathbb{N}^\star$, there exists a constant $C_{\alpha, \beta}$ such that
\begin{align*}
	| \partial_x^\alpha \partial_\xi^\beta a (x, \xi) | \le C_{\alpha, \beta} \left< \xi \right>^{m-\beta}, \quad \forall (x, \xi) \in \mathbb{T} \times \mathbb{R}.
\end{align*}
 For any symbol $a \in \Gamma^m$, we associate an operator acting on periodic functions 
\begin{align*}
 u:=u(x)= \sum_{j \in \mathbb{Z}} u_j e^{ijx} \mapsto	\left(\mathrm{Op}(a)u\right)(x):= \sum_{j \in \mathbb{Z}} a(x,j) u_j e^{jx}. 
\end{align*}
We shall also consider families of pseudo-differential operators and symbols depending on parameters $(\varphi, \omega) \in \mathbb{T}^\nu \times \mathbb{R}$ of the form 
\begin{align*}
	a : (\omega, \varphi) \rightarrow a (\omega, \varphi) \in \Gamma^m. 
\end{align*} 
Thus, for any $(\omega, \varphi)$, $\mathrm{Op}(a(\omega, \varphi))$ is an operator acting on periodic functions defined on $\mathbb{T}^\nu \times \mathbb{R}$. Moreover, at fixed $\omega$,  the family of operators $\mathrm{Op}(a(\omega, \varphi)$ naturally defines an operator acting on periodic functions defined on $\mathbb{T}^{\nu+1}$, denoted $\mathrm{Op}(a)$. For instance, for any $u \in H^s(\mathbb{T}^{\nu+1})$, $\mathrm{Op}(a)u: \mathbb{T}^{\nu+1} \rightarrow \mathbb{C}$ is defined by
\begin{align*}
	\mathrm{Op}(a)u : (\varphi, x) \rightarrow \mathrm{Op}(a(\omega, \varphi))u(\varphi, \cdot). 
\end{align*}
For a family of symbols $a(\omega, \varphi)$, we sometimes use the notation $a(\omega; \varphi, x, \xi):=a(\omega, \varphi)(x, \xi)$. Moreover, for $\omega \in \mathcal{O}$, we write $a(\omega; \,\cdot\,)$ for the function $(\varphi, x, \xi) \rightarrow a( \omega; \varphi, x, \xi)$. This notation will also be used when we fix other variables, for instance, for fixed $(\omega, \xi) \in \mathcal{O} \times \mathbb{R}$, we will consider the function $a(\omega;\, \cdot\,, \,\cdot\,, \xi)$. \\ \\
The set of such (families of) symbols which are Lispchitz in $\omega \in \mathcal{O}$ and valued in $\Gamma^m$ is denoted by $S^{m}$. For any $p \in \mathbb{N}^\star$, $s \ge s_0$, we defined a norm on $S^m$ by
$$
\Vert a \Vert^{\gamma, \mathcal{O}}_{m, s, p} := \sup_{\omega \in \mathcal{O}} \Vert a(\omega; \cdot ) \Vert_{m, s, p} + \gamma \sup_{\omega_1 \neq \omega_2 \\ \omega_1, \omega_2 \in \mathcal{O}} \frac{\Vert a(\omega_1; \cdot)-a(\omega_2; \cdot) \Vert^{\gamma, \mathcal{O}}_{m, s-1, p}}{| \omega_1- \omega_2|},
$$ 
where
$$
\Vert a(\omega; \cdot ) \Vert_{m, s, p}:= \max_{0 \le \beta \le p} \sup_{\xi \in \mathbb{R}} \Vert \partial_{\xi}^\beta a ( \omega; \cdot, \cdot, \xi)\Vert_s \left< \xi \right>^{-m+\beta}.
$$
In the first order reduction of Section \ref{se:redfirsto}, we shall need matrices of pseudo-differential operators. Thus, given a matrix of symbols in $S^m$
$$
A:= A( \omega, \varphi, x, \xi) := \left(\begin{array}{cc} a& b \\
c & d
\end{array} \right) \in S^m \otimes \mathcal{M}_2(\mathbb{C}), 
$$
associated with the norm
$$
\Vert A \Vert_{m, s, p}^{\gamma, \mathcal{O}} := \max_{f=a,b,c,d}  \Vert f \Vert_{m,s,p}  
$$
and the operator, acting on $H^s(\mathbb{T}, \mathbb{C}^2)$ (and $H^s(\mathbb{T}^{\nu+1}, \mathbb{C}^2)$ by extension)
$$
\mathrm{Op}(A):= \left(\begin{array}{cc}\mathrm{Op}( a)& \mathrm{Op}(b) \\
\mathrm{Op}(c )& \mathrm{Op}(d)
\end{array} \right) \in S^m \otimes \mathcal{M}_2(\mathbb{C}).
$$
By a small abuse of notation, for a pseudo-differential operator $B:= \mathrm{Op}(b)$, with $b \in S^m$, we shall sometimes write $B \in S^m$ or $\Vert B \Vert_{m,s, p}$, identifying $B$ and its symbol $b$. \\ \\
Recalling the algebraic properties of operators introduced in Definition \ref{def:parityoper}, a symbol is said to be reversible, reversibility preserving or parity preserving if its associated pseudo-differential operator is reversible, reversibility preserving or parity preserving respectively. \\ \\
We recall the following lemma.
\begin{lemma}[Lemma 3.10 in \cite{berti2024reducibility}]
 A symbol $a \in S^m$, $m \in \mathbb{R}$, is
\begin{enumerate}[leftmargin=5.5mm]
 \item reversible if and only if $$a(-\varphi, x, \xi)=-\overline{a(\varphi, x, -\xi)},$$
\item reversibility preserving if and only if
 $$a(-\varphi, x, \xi)=\overline{a(\varphi, x, -\xi)},$$
 \item parity preserving if and only if
 $$a(\varphi, -x, -\xi)=a(\varphi, x, \xi),$$
 \item A matrix of symbols $A \in S^m \otimes \mathcal{M}_2(\mathbb{C})$ is real-to-real if and only if it has the form
 $$
 A (\varphi, x, \xi)= \left( \begin{array}{cc}a (\varphi, x, \xi) & b(\varphi, x, \xi) \\ \overline{b(\varphi, x, -\xi)} & \overline{a(\varphi, x, -\xi)} \end{array} \right).
 $$
 Then, $A$ is reversible, reversibility or parity preserving if $a$ and $b$ are reversible, reversibility or parity preserving.
\end{enumerate}
\end{lemma}
We need the following standard statements concerning pseudo-differential operators.
\begin{lemma}[Action on Sobolev spaces, Lemma 3.3 in \cite{berti2024reducibility} or Lemma 2.21 in \cite{MR4062430}] {}\label{sobaction}
Let  $m\ge0$ and $a\in S^m$. 
Then for any $s\geq s_0$
and any $u\in H^{s+m}(\mathbb{T}^{\nu+1}, \mathbb{C})$, 
$$
\Vert \mathrm{Op}(a)u \Vert_{s}^{\gamma, \mathcal{O}} \lesssim \Vert a \Vert^{\gamma, \mathcal{O}}_{m, s_0, 0} \Vert u \Vert_{m+s}^{\gamma, \mathcal{O}} + \Vert a \Vert^{\gamma, \mathcal{O}}_{m, s, 0} \Vert u \Vert_{m+s_0}^{\gamma, \mathcal{O}} 
$$
\end{lemma}
\noindent 
Given $a \in S^m$ and $b \in S^{m'}$, we denote by $a \# b $, the symbol of the composition operator $\mathrm{Op}(a) \circ \mathrm{Op}(b)$. One has the following result.
\begin{lemma}[Lemma 3.4 in \cite{{berti2024reducibility}}]  \label{lem:pseudocom}
One has $a\#b \in S^{m+m'}$ and, for any $p \in \mathbb{N}, s \ge s_0$, 
$$
\Vert a \# b \Vert^{\gamma, \mathcal{O}}_{m+m',s,p} \lesssim \Vert a \Vert^{\gamma, \mathcal{O}}_{m,s,p} \Vert b  \Vert^{\gamma, \mathcal{O}}_{m',s_0+p+|m|,p} +\Vert  a \Vert^{\gamma, \mathcal{O}}_{m,s_0,p} \Vert b  \Vert^{\gamma, \mathcal{O}}_{m',s+p+|m|,p}. 
$$
\end{lemma}
\noindent
We also need the following lemma on the commutator of $a$ and $b$, $\mathrm{Op}(a) \mathrm{Op}(b) -\mathrm{Op}(b) \mathrm{Op}(a)$, of symbol $a \star b= a \#b - b\#a$. 
\begin{lemma}
For any $s \in S^m$, $b \in S^{m'}$, $m, m' \in \mathbb{R}$, one has $a \star b \in S^{m+m'-1}$ and 
$$
\Vert a \star b \Vert^{\gamma, \mathcal{O}}_{m+m'-1,s,p} \lesssim \Vert a \Vert^{\gamma, \mathcal{O}}_{m,s+2+|m'|+p,p+1} \Vert b  \Vert^{\gamma, \mathcal{O}}_{m',s_0+2+|m|+p,p+1} +\Vert  a \Vert^{\gamma, \mathcal{O}}_{m,s_0+2+|m'|+p,p+1} \Vert b  \Vert^{\gamma, \mathcal{O}}_{m',s+2+|m|+p,p+1}. 
$$
\end{lemma}

\section{Null coordinates for quasi-periodic $1+1$ Lorentzian metric} \label{se:ncqlm}
\noindent
In this section, we prove Theorems \ref{th:redumetric} and \ref{th:redumetricws}. We consider a metric $g$ on $\mathbb{R}_t \times \mathbb{S}^1_x$ of the form
\begin{eqnarray} \label{eq:metric}
g = -  A dt \otimes dt+ A^{-1} dx \otimes dx, 
\end{eqnarray}
with $A$ quasi-periodic i.e. 
\begin{equation} \label{def:mathcalA}
A(t,x):= \mathcal{A}(\omega t,x), 
\end{equation}
where $\omega \in \mathbb{R}^\nu$, $\nu \in \mathbb{N}^\star$, and $\mathcal{A}: \mathbb{T}^{\nu+1} \rightarrow (0, + \infty)$. We will first prove Theorem \ref{th:redumetric} which concerns the case where $\mathcal{A}$ is even in both $\varphi$ and $x$. In Section \ref{se:redumetricws}, we then revisit the proof in the case where these extra assumptions do not hold. 
\subsection{Reduction to transport equations on the torus}
We want to construct new coordinates $(u,v)$ verifying 
\begin{eqnarray}
	\partial_{t} u + A \partial_{x}u =0, \quad  \partial_{t}v - A \partial_{x}v =0. \label{eq:uv}
\end{eqnarray}
Indeed, assuming we can solve these equations and now define $\left(\tau(t,x), R(t,x)\right)$ via 
\begin{align} \label{def:tauR}
	\begin{dcases}
	\frac{v+u}{2}= \frac{\tau}{\alpha}, \\
	\frac{v-u}{2}=R,
\end{dcases} 
\end{align}  
where $\alpha:=\alpha(\omega) > 0$ is an adjustable real parameter, independent of $(\varphi,x)$. We see that, in $(\tau, R)$ coordinates, the metric then takes the form
\begin{equation}\label{eq:metrictauR}
g = \Omega^2 \left( -\alpha^{-2} d \tau \otimes d\tau + dR \otimes dR \right),
\end{equation}
where
\begin{equation}\label{eq:Omega}
\Omega^2 = - \frac{1}{A \partial_{x} u \partial_{x}v},
\end{equation}
which is of the form \eqref{eq:gmetrictauR}.
%
For the flat metric, 
\begin{align*}
	\eta= - d t \otimes dt + dx \otimes dx,
\end{align*}
the null coordinate equations \eqref{eq:uv} reduce to $\partial_{t} u +\partial_{x}u =0$, $\partial_{t}v -\partial_{x}v =0$, the solutions of which are given by the pair of functions of the form $(f (t-x), g(t+x))$, with the simplest one being just $(t-x, t+x)$. Since we are in a perturbative regime near flat space, we wish to construct null coordinates which are close to $(t-x, t+x)$. \\ \\
To solve the $(u,v)$ equations \eqref{eq:uv}, we thus consider another pair of unknowns $(U(t,x),V(t,x))$ as 
\begin{equation}\label{def:UV}
U(t,x) = u(t,x)-\frac{t}{\rho}+x, \quad V(t,x)= v(t,x)-\frac{t}{\rho}-x,
\end{equation}
where $\rho:=\rho(\omega) >0 $ is an adjustable parameter, so that  from \eqref{eq:uv}, $(U,V)$ must solve the inhomogeneous transport equations
\begin{align}
	\partial_{t}U +A \partial_{x} U &= A- \frac{1}{\rho}, \label{eq:U}\\
\partial_{t} V -A \partial_{x} V &= A- \frac{1}{\rho}. \label{eq:V}
\end{align} 
Although there exist many solutions depending on initial conditions, as we will see below, there are unique (up to constants) solutions to the $U$-equation \eqref{eq:U} and $V$-equation \eqref{eq:V} among quasi-periodic solutions of time frequency $\omega \in \mathbb{R}^\nu$, provided $\omega$ verifies diophantine type conditions. 

\begin{remark}
Up to signs, $(U,V)$ corresponds to the functions $(\beta_+, \beta_-)$ introduced in Lemma 7.5 of \cite{berti2024reducibility}. In \cite{berti2024reducibility}, there are used to straighten the vector fields corresponding to the two directions of propagation, somehow independently from each other. Here, we shall use them together, since we eventually construct a unique coordinate system $(\tau, R)$ based on the functions $(U,V)$.  This is to be compared with having two coordinate systems $(t, y_+:= x + \beta_+(\omega t, x))$ and $(t, y_-:= x + \beta_-(\omega t, x))$. While the $(t, y_\pm)$ coordinates only straighten one of the two vector fields $\partial_t \pm A \partial_x$, the $(u,v)$ or $(\tau, R)$ coordinates straighten the entire metric (up to a conformal factor, cf~\eqref{eq:metrictauR}). The (small in our opinion) price to pay is that the time variable, or equivalently $\varphi$, must also be transformed, besides the spatial variable. 
Note also that, writing the Klein-Gordon equation \eqref{eq:KG} as a first order system (even using only classical differential operators) of two couple transport equations and then using the two coordinate systems $(t, y_-:= x + \beta_-(\omega t, x))$ to straighten the first order differential operators, necessarily induce \emph{nonlocality} in the equations because of the coupling between the two transport equations written in two distinct coordinate systems. 
\end{remark}
\noindent
As in Lemma 7.5 of \cite{berti2024reducibility}, we note that $V$ can be obtained from $U$ provided that $A$ is time symmetric or even in $x$.
\begin{lemma} \label{lem:parityTR}
\begin{enumerate}[leftmargin=5.5mm]
\item Assume that $A$ is even in $t$, $A(-t,x)=A(t,x)$. Then, $W(t,x):=- U(-t,x)$ verifies the $V$-equation \eqref{eq:V}. 
\item Assume that $A$ is even in $x$, $A(t,-x)=A(t,x)$. Then, $Q(t,x):=U(t,-x)$ verifies the $U$-equation \eqref{eq:U}. 
\end{enumerate}
\begin{proof}
One computes, using the symmetry assumptions,  
\begin{align*}
	\partial_t W(t,x)- A(t,x) \partial_x W(t,x) &=  \partial_t U (-t,x)+A(-t,x) \partial_x U(-t,x)= A(-t,x)- \frac{1}{\rho}=A(t,x) - \frac{1}{\rho} \\
\partial_t Q(t,x)- A(t,x) \partial_x Q(t,x)&= \partial_t U(t,-x)+A(t,-x)  \partial_x U (t,-x)= A(t,-x)- \frac{1}{\rho}=A(t,x) - \frac{1}{\rho}.
\end{align*} 
that completes the proof.
\end{proof}
\end{lemma}
\noindent
The functions $U$ and $V$ will be constructed as quasi-periodic functions. Thus, we look for $\mathcal{U}$ and $\mathcal{V}$ such that
\begin{align}
	\mathcal{U}(\omega t, x):&= U(t,x),\label{def:mathcalU} \\
\mathcal{V}(\omega t, x):&= V(t,x), \label{def:mathcalV}
\end{align} 
solving 
\begin{align}
	\omega \cdot \partial_\phi \mathcal{U} +\mathcal{A}\,\mathcal{U}_x &=\mathcal{A}- \frac{1}{\rho},\label{eq:mathcalU} \\
\omega \cdot \partial_\phi \mathcal{V} -\mathcal{A}\, \mathcal{V}_x&= \mathcal{A}- \frac{1}{\rho}, \label{eq:mathcalV}
\end{align} 
where $\mathcal{A}$ was introduced in \eqref{def:mathcalA}. 
\begin{remark}\label{rem:symUV}
Within the class of regular quasi-periodic functions, since the $U$ and $V$ equation are actually solvable uniquely up to constants, cf~Remark \ref{rem:uniq}, it follows from Lemma \ref{lem:parityTR} that $V(t,x)=U(t,-x)+c_1=-U(-t,x)+c_0$, assuming that $A$ is even in $t$ and $x$. In particular, we have that 
\begin{enumerate}[leftmargin=5.5mm]
\item $U+V$ is even with respect to $x$ and odd with respect to $t$. 
\item $V-U$ is odd with respect to $x$ and even with respect to $t$. 
\end{enumerate}
Moreover, the same symmetry holds true for the functions $\mathcal{U}$ and $\mathcal{V}$, for instance $\mathcal{U}+\mathcal{V}$ is even in $x$, odd in $\varphi$, for the same algebraic reasons. 
\end{remark}
\noindent
Let $1+a_0= \mathcal{A}$, so that $a_0$ is as in Proposition \ref{prop:redvf} and assumption \eqref{ineq:Asmallness} implies \eqref{prop:reduassum}. Up to an overall sign, the equation for $\mathcal{U}$ is then the same as for the function $\beta$ of Proposition \ref{prop:redvf}. This guarantess the existence of $\mathcal{U}$ provided \eqref{ineq:Asmallness} holds and implies, as in Proposition \ref{prop:redvf}, that the function $\mathcal{U}$ straightens the transport operator 
\begin{align*}
	L := \omega \cdot \partial_\varphi  +\mathcal{A} \partial_x.
\end{align*}
Indeed, with
 $$\Psi: (\varphi, x) \rightarrow \left(\varphi, y:=x- \mathcal{U}(\varphi, x) \right), $$ 
and assuming that $\mathcal{U}$ is small enough so that $\Psi$ is a diffeormorphism of $\mathbb{T}^{\nu+1}$, one has, for a function $g:=g(\varphi, y)$, denoting $f(\varphi, x)=g(\varphi, x -\mathcal{U}(\varphi, x))=g \circ \Psi$, that
\begin{align*}
 L (g \circ \Psi) \circ \Psi^{-1} &= \left( \omega \cdot \partial_{\varphi} f + \mathcal{A} \partial_x f  \right) \circ  \Psi^{-1} \\
&=  \omega\cdot \partial_\varphi g -\left([ \omega \cdot \partial_\varphi \mathcal{U}] \circ  \Psi^{-1}\right) \partial_y g + \left(\mathcal{A}  \circ  \Psi^{-1} \right)\partial_y g - \left([\mathcal{A} \partial_x \mathcal{U} ]  \circ  \Psi^{-1}\right)\partial_y g \\
&=  \omega \cdot \partial_\varphi g - \left(\omega \cdot \partial_\varphi \mathcal{U}+\mathcal{A} \partial_x \mathcal{U}-\mathcal{A}+ \frac{1}{\rho} \right) \circ  \Psi^{-1} \, \partial_y g +\frac{1}{\rho} \partial_y g \\
&= \omega \cdot\partial_\varphi g + \frac{1}{\rho} \partial_y g.
\end{align*}
In particular, we have 
\begin{align*}
 L \left(\left( \mathcal{U} \circ \Psi^{-1} \right) \circ \Psi \right) \circ \Psi^{-1}= \left(L \mathcal{U}\right)\circ \Psi^{-1} = A \circ \Psi^{-1} - \frac{1}{\rho}= \left( \omega \partial_\varphi + \frac{1}{\rho} \partial_x \right) \left(\mathcal{U} \circ \Psi^{-1}\right).
\end{align*}
Since this last term is an exact constant coefficient derivative on the torus, it follows that it has zero mean over $\mathbb{T}^{\nu+1}$ and thus, so does $A \circ \Psi^{-1} - \frac{1}{\rho}$,  i.e. 
$$
\frac{1}{\rho}:= \left< A \circ \Psi^{-1} \right>_{\varphi, y}.
$$
This fixes the value of  parameter $\rho$ introduced in \eqref{def:UV}. 
Note also that, assuming that $a_0=\mathcal{A}-1$ verifies the assumption of Proposition \ref{prop:redvf}, we see that 
$$\frac{1}{\rho}= m_\infty(\omega),$$ 
with $m_\infty(\omega)$ as in Proposition \ref{prop:redvf}. In particular, Proposition \ref{prop:redvf} also guarantees the (Lipschitz) regularity of $\rho$ as function of $\omega$. 
We summarize this below.
\begin{prop}\label{prop:main1}
Let $\gamma \in (0,1)$, $\iota := \nu+3$, $\mathcal{O}_0 \subset \mathbb{R}^\nu$ compact of postitive Lebesgue measure, and let  $s_1 \ge 2\iota + s_0+4$ and $\eta_\star(s_1) > 0$ be as in Proposition \ref{prop:redvf}. Assume that $\mathcal{A}$ is even in $\varphi$ and $x$ and verifies  
\begin{equation} \label{assump:A}
\gamma^{-1} \Vert  \mathcal{A}-1 \Vert _{s_1}^{\gamma, \mathcal{O}_0} \le  \eta_\star, \\
\end{equation} 
Then, there exists a Lipschitz function $m : \mathcal{O}_0 \rightarrow \mathbb{R}$ such that 
\begin{itemize}[leftmargin=5.5mm]
\item $m-1$ has small Lipschitz norm
$$
|m_{\infty} (\omega)-1|^\gamma  \le \gamma \delta,
$$
\item In the set 
$$
\mathcal{O}^{2\gamma}_\infty := \left\{ \omega \in \mathcal{O}_0: | \omega \cdot \ell + m_{\infty} (\omega) \cdot j | > \frac{2\gamma}{\left< \ell, j \right>^\iota}, \,\,\forall (l, j ) \in \mathbb{Z}^{\nu+1}\setminus \{0 \} \right\},
$$ the following holds. 
\begin{enumerate}[leftmargin=5.5mm]
\item The equations 
$$
\omega \cdot \partial_\varphi \mathcal{U} +A \mathcal{U}_x = A- m, \quad \omega \cdot \partial_{\varphi} \mathcal{V} -A \mathcal{V}_x = A- m,
$$
admit $C^\infty$ solutions, unique up constants, and we can fix $\mathcal{U}$ and $\mathcal{V}$ to be the solutions with zero mean over $\mathbb{T}^{\nu+1}$.
\item $\left(\mathcal{U}, \mathcal{V}\right)$ verifies the following algebraic properties: 
\begin{eqnarray*}
\mathcal{V}(\varphi,x)=- \mathcal{U}(-\varphi,x), \quad \mathcal{V}(\varphi, x)=\mathcal{U}(\varphi, -x). 
\end{eqnarray*}
\item  $\mathcal{U}$ and $\mathcal{V}$ verifies the estimate \label{item:Ues}
 \begin{eqnarray} \label{es:UV}
 \Vert  \mathcal{U}  \Vert_{s}^{\gamma, \mathcal{O}_\infty^{2\gamma} } +
  \Vert  \mathcal{V} \Vert_{s}^{\gamma, \mathcal{O}_\infty^{2\gamma} } \lesssim_s \gamma^{-1}
\Vert \mathcal{A}-1 \Vert_{s+2\iota+4}^{\gamma, \mathcal{O}_0}, \quad \forall s \ge s_0,\quad  s \in \mathbb{N}.
\end{eqnarray}
\end{enumerate}
\end{itemize}
\end{prop}

\subsection{Diffeomorphism of the cylinder and diffeomorphism of the torus}
Provided that $\mathcal{U}$, $\mathcal{V}$ are small enough, it also follows from the previous proposition that we can define the change of coordinates $(t,x) \rightarrow (\tau, R)$ as explained in the next proposition. 
\begin{prop}\label{prop:main2}
Under the assumptions of the Proposition \ref{prop:main1}, there exists $\eta_2$ such that, if 
\begin{equation} \label{assump:A2}
\gamma^{-1} \Vert \mathcal{A}-1 \Vert_{s_0+2\iota+5}^{\gamma, \mathcal{O}_0} \le  \eta_2,
\end{equation}
then, defining $U(t,x)=\mathcal{U}(\omega t,x)$, $V(t,x)= \mathcal{U}(\omega t,x)$ and  $(u,v)$ by \eqref{def:UV} and then $(\tau, R)$ by \eqref{def:tauR} with $\rho=\alpha=m^{-1}$, the map $\mathcal{C}:(t, x) \rightarrow (\tau, R)$ is a diffeomorphism of $\mathbb{R} \times \mathbb{S}^1$ and, in $(\tau,R)$ coordinates, the metric takes the form \eqref{eq:metrictauR} with $\alpha=m^{-1}$. Moreover, $\tau$ is even in $t$, odd in $x$ and $R$ is odd in $x$, even in $t$. 
\end{prop}
\begin{proof}
For $\omega \in \mathcal{O}^{2\gamma}_\infty$, we thus consider the map
\begin{align*}
\mathcal{C}: \mathbb{R}_t \times \mathbb{S}^1_x &\rightarrow  \mathbb{R}_\tau \times \mathbb{S}^1_R \\
(t,x) &\rightarrow  \left(\tau(t,x), R(t,x)\right) :=\left( t + \frac{1}{2m}\left( U(t,x)+ V(t,x) \right) , x + \frac{V(t,x)-U(t,x)}{2}\right). 
\end{align*}
Note that $\mathcal{C}$ is well-defined since $U$ and $V$ are $2\pi$-periodic with respect to $x$, so that $R(t,x+2\pi)=R(t,x)+2\pi \equiv R(t,x) \bmod 2\pi$. The fact that $\mathcal{C}$ is a local and then global diffeormorphism easily follows from the uniform bounds
\begin{align*}
	\Vert U  \Vert_{L^\infty(\mathbb{R}_t \times \mathbb{S}^1_x) }+
	\Vert   V  \Vert_{L^\infty(\mathbb{R}_t \times \mathbb{S}^1_x) }+
	\Vert   \partial U  \Vert_{L^\infty(\mathbb{R}_t \times \mathbb{S}^1_x) }+
	\Vert \partial V \Vert_{L^\infty(\mathbb{R}_t \times \mathbb{S}^1_x) }< \frac{1}{8},
\end{align*}
 which in turn follow from item (\ref{item:Ues}) of Proposition \ref{prop:main1}, \eqref{assump:A2} and the Sobolev embedding \eqref{ineq:sob}. 
Finally, the algebraic properties of $\tau$ and $R$ follows from Remark \ref{rem:symUV}.
\end{proof}
\noindent
From now on, we relabel $\eta_\star$ and $s_1$ so that \eqref{assump:A} implies \eqref{assump:A2}. \\ \\
\noindent
To understand how quasi-periodic functions transform under the change of coordinates, we need to lift the $\mathcal{C}$ map to a map defined on the torus $\mathbb{T}^{\nu+1}$. Consider for instance, a quasi-periodic function with respect to $\tau$ of the form
\begin{align*}
	B(\tau,R):= \mathcal{B}( \omega \tau, R),
\end{align*}
where $\mathcal{B}$ is defined on $\mathbb{T}^{\nu+1}$.
Then, we have with respect to $(t,x)$,
\begin{align*}
B(\tau(t,x), R(t,x)) &= B \left( \mathcal{C}(t,x)  \right)\\
&= \mathcal{B}\left( \omega  t + \frac{\omega}{2 m(\omega)} \left( U(t,x)+V(t,x) \right) , x + \frac{1}{2} \left( V(t,x)-U(t,x) \right) \right) \\
&=\mathcal{B}\left( \omega  t +  \frac{\omega}{2 m(\omega)} \left( \mathcal{U}(\omega t,x)+  \mathcal{V}(\omega t,x)\right) , x + \frac{1}{2} \left( \mathcal{V}(\omega t,x)-\mathcal{U}(\omega t,x) \right) \right) \\
&= \mathcal{B} \circ \Psi( t\omega, x), 
\end{align*}
where $\Psi:=\Psi_\omega$ is the ($\omega$-dependent) map
\begin{align} \label{def:Psi}
\Psi: \mathbb{T}^{\nu+1} &\rightarrow  \mathbb{T}^{\nu+1} \nonumber\\
(\varphi, x) &\rightarrow \Psi(\varphi, x) := \left( \varphi'(\varphi, x), y (\varphi, x)\right) \nonumber \\
& \quad\quad\quad\quad\quad = \left(\varphi + \frac{\omega}{2m(\omega)}\left( \mathcal{U}(\varphi, x) + \mathcal{V}(\varphi, x)\right) , x + \frac{1}{2} \left( \mathcal{V}(\varphi, x) - \mathcal{U}(\varphi, x)\right) \right).
\end{align}

\begin{remark}
It is interesting to note that the new variable $\varphi'$ is a perturbation of $\varphi$ along the straight line parallel to the frequency vector $\omega$.
\end{remark}
\noindent
As a direct application of item $(i)$ of Lemma \ref{lem:difftor}, we obtain that $\Psi$ is a diffeomorphism.

\begin{lemma} \label{lem:psidiff}Under the assumption of Proposition \ref{prop:main2}, there exists $\eta_3$ such that, if 
\begin{equation} \label{assump:A3}
\gamma^{-1} \Vert \mathcal{A}-1 \Vert_{s_0+2\iota+5}^{\gamma, \mathcal{O}_0} \le  \eta_3,
\end{equation}
$\Psi$ is a diffeomorphism, $\Psi^{-1}(\varphi',y)=(\varphi', y) +q(\varphi',y)$, with $q(\varphi', y):=q_\omega(\varphi',y) \in W^{s,\infty}$ verifying
\begin{align*}
	\Vert q \Vert^{\gamma, \mathcal{O}_\infty^{2\gamma} }_ {W^{s,\infty}} \lesssim \sup_{\omega \in \mathcal{O}_0} (|\omega| +1) \, \Vert \mathcal{U} \Vert^{\gamma, \mathcal{O}_\infty^{2\gamma} }_{W^{s,\infty}} \lesssim \sup_{\omega \in \mathcal{O}_0} (|\omega| +1) \, \Vert \mathcal{U} \Vert^{\gamma, \mathcal{O}_\infty^{2\gamma} }_{s+s_0} \lesssim_{\mathcal{O}_0} \Vert \mathcal{A}-1 \Vert^{\gamma, \mathcal{O}_\infty^{2\gamma} }_{s+s_0+2 \iota+4}.
\end{align*}
\end{lemma}
\noindent
From now on, we relabel $\eta_\star$ so that  \eqref{assump:A} implies \eqref{assump:A3}. \\ \\
From item \eqref{es:difftor} of Lemma \ref{lem:difftor}, it also follows that, under the assumption of Proposition \ref{prop:main2}, for any function $h \in H^s(\mathbb{T}^{\nu+1})$, one also has 
\begin{eqnarray} \label{es:tamepsi1}
\Vert h \circ \Psi \Vert_s &\le& \Vert h \Vert_s + C \left( \Vert \mathrm{Id}- \Psi \Vert_{W^{1,\infty}}   \Vert h \Vert_s+ \Vert\mathrm{Id}-\partial \Psi \Vert_{W^{s-1, \infty}}\Vert h \Vert_1 \right), \\
\Vert h \circ \Psi-h \Vert_s &\le& C \left( \Vert \mathrm{Id}-\Psi \Vert_{L^\infty} \Vert h \Vert_{s+1} + \Vert\mathrm{Id}-\Psi \Vert_{W^{s,\infty}}\Vert h \Vert_2 \right).\label{es:tamepsi12}
\end{eqnarray}
Moreover, the same estimates hold with $\Psi$ replaced by $\Psi^{-1}$ and we have the additional estimates, for any Lipschitz family of $H^s(\mathbb{T}^{\nu+1})$-functions $h:=h_\omega$,
\begin{align}
	\Vert h \circ \Psi \Vert^{\gamma,\mathcal{O}^{2\gamma}_\infty}_s &\leq \Vert h \Vert^{\gamma,\mathcal{O}^{2\gamma}_\infty}_s + C \left( \Vert \mathrm{Id}- \Psi \Vert^{\gamma,\mathcal{O}^{2\gamma}_\infty}_{W^{1,\infty}}   \Vert h  \Vert^{\gamma,\mathcal{O}^{2\gamma}_\infty}_s+ \Vert\mathrm{Id}-\Psi \Vert^{\gamma,\mathcal{O}^{2\gamma}_\infty}_{W^{s, \infty}}\Vert h \Vert^{\gamma,\mathcal{O}^{2\gamma}_\infty}_2 \right), \label{es:tamepsilip1} \\  
\Vert h \circ \Psi-h \Vert^{\gamma,\mathcal{O}^{2\gamma}_\infty}_s & \leq C \left( \Vert \mathrm{Id}-\Psi \Vert^{\gamma,\mathcal{O}^{2\gamma}_\infty}_{W^{1,\infty}} \Vert h \Vert^{\gamma,\mathcal{O}^{2\gamma}_\infty}_{s+1} + \Vert\mathrm{Id}-\Psi \Vert^{\gamma,\mathcal{O}^{2\gamma}_\infty}_{W^{s,\infty}}\Vert h \Vert^{\gamma,\mathcal{O}^{2\gamma}_\infty}_3 \right).\label{es:tamepsilip12}
\end{align} 
%
%
\noindent
In view of the definition of $\Psi$, we have 
$$
(\mathrm{Id}- \Psi)(\varphi,x)= \left( \frac{\omega}{2m(\omega)}\left( \mathcal{U}(\varphi, x) + \mathcal{V}(\varphi, x)\right) , \frac{1}{2} \left( \mathcal{V}(\varphi, x) - \mathcal{U}(\varphi, x)\right) \right),
$$
so that we can substitute the norms of $\mathcal{U}$ and $\mathcal{V}$ instead of those of $\mathrm{Id}-\Psi$ in the above estimate.  \\ \\
Recall that by item (\ref{item:Ues}) of Proposition \ref{prop:main1}, we have  
\begin{align*}
	\Vert \mathcal{U}  \Vert_{s}^{\gamma, \mathcal{O}_{\infty}^{2\gamma}}+
	\Vert  \mathcal{V} \Vert_{s}^{\gamma, \mathcal{O}_{\infty}^{2\gamma}} &\lesssim \gamma^{-1} \Vert \mathcal{A}-1\Vert^{\gamma, \mathcal{O}_0}_{s+2\iota+4}.
\end{align*}
 Combined with \eqref{es:tamepsilip1} and \eqref{es:tamepsilip12}, we obtain the tame estimates for the change of coordinates. 
\begin{lemma} \label{lem:psiest}
For any Lipschitz family of $H^s(\mathbb{T}^{\nu+1})$ functions $h:=h_\omega$, one has
\begin{align}\label{es:tamepsi2}
\begin{split}
& \Vert h \circ \Psi \Vert^{\gamma, \mathcal{O}_{\infty}^{2\gamma}}_s \leq \Vert h \Vert^{\gamma, \mathcal{O}_{\infty}^{2\gamma}}_s 
	\hbox{}+ C \left( \gamma^{-1} \Vert \mathcal{A}-1\Vert^{\gamma, \mathcal{O}_0}_{s_0+2\iota+5}  \Vert h \Vert^{\gamma, \mathcal{O}_{\infty}^{2\gamma}}_s+  \gamma^{-1} \Vert \mathcal{A}-1\Vert^{\gamma, \mathcal{O}_0}_{s+s_0+2\iota+4} \Vert h  \Vert^{\gamma, \mathcal{O}_{\infty}^{2\gamma}}_2 \right), \\
& \Vert h \circ \Psi-h \Vert^{\gamma, \mathcal{O}_{\infty}^{2\gamma}}_s \leq C \left( \gamma^{-1} \Vert \mathcal{A}-1\Vert^{\gamma, \mathcal{O}_0}_{s_0+2\iota+5}  \Vert h \Vert^{\gamma, \mathcal{O}_{\infty}^{2\gamma}}_{s+1} +\gamma^{-1} \Vert \mathcal{A}-1 \Vert^{\gamma, \mathcal{O}_0}_{s+s_0+2\iota+4}  \Vert h \Vert^{\gamma, \mathcal{O}_{\infty}^{2\gamma}}_3 \right).
\end{split}
\end{align} 
\end{lemma}

\begin{remark}
In view of the smallness assumption \eqref{assump:A} and $\gamma^{-1} || \mathcal{A}-1||_{s_0+2\iota+5} \le 1$, we have 
\begin{align*}
	\Vert h \circ \Psi \Vert_s \lesssim \Vert h \Vert_s +  C  \gamma^{-1} \Vert \mathcal{A}-1\Vert_{s+\mu} \Vert h \Vert_1, 
\end{align*} 
with $\mu$ independent of $s$. 
\end{remark}
\begin{proof}
We just apply the Sobolev embedding \eqref{ineq:sob} to control the $L^\infty$ norms in \eqref{es:tamepsi1} as follows
\begin{align*}
	\Vert \mathcal{U} \Vert_{L^\infty}  &\lesssim  \Vert\mathcal{U}\Vert_{s_0}\lesssim \gamma^{-1} \Vert \mathcal{A}-1\Vert_{s_0+2\iota+4}, \\
\Vert \mathcal{U} \Vert_{W^{1,\infty}} & \lesssim  \Vert\mathcal{U}\Vert_{1+s_0}\lesssim \gamma^{-1} \Vert \mathcal{A}-1\Vert_{s_0+2\iota+5}, \\
\Vert \partial \mathcal{U} \Vert_{W^{s-1,\infty}} &\lesssim \Vert\mathcal{U}\Vert_{s+s_0} \lesssim \gamma^{-1} \Vert \mathcal{A}-1\Vert_{s+s_0+2\iota+4},
\end{align*} 
that completes the proof.
\end{proof}

\subsection{Estimates on the conformal factor}
From the estimates on $\mathcal{U}$, $\mathcal{V}$ and $\Psi$, one can also compute the regularity of the conformal factor $\Omega$ with respect to the new coordinates. 
\begin{lemma} \label{lem:confreg}
Let 
\begin{align*}
	\Omega^2 (\tau,R)= -\frac{1}{A \partial_{x}u \partial_{x}v} \circ \Psi^{-1}(\tau,R)
\end{align*}
be the conformal factor associated to the metric in $(\tau,R)$ coordinates as in \eqref{eq:metrictauR}-\eqref{eq:Omega}. Then, 
\begin{align*}
	\Omega^2 \circ \Psi (t,x)= -\frac{1}{{\mathcal{A}( \omega t, x )( \partial_{x} \mathcal{U}( \omega t, x ) -1)(\partial_{x}\mathcal{V}( \omega t, x )+1)}}  
\end{align*}
and the function
\begin{align*}
	(\varphi,x)\mapsto \frac{1}{{\mathcal{A}(\varphi,x)(\partial_{x}\mathcal{U} (\varphi,x)-1)( \partial_{x}\mathcal{V}(\varphi,x) +1)}}
\end{align*}
is independently even in both $\varphi$ and $x$. Moreover, $\Omega^2(\tau,R)$ is quasi-periodic with respect to $\tau$ of frequency vector $\omega$ and, denoting $\Omega(\tau,R)=\mathfrak{O}(\omega \tau,R)$, we have the estimate
\begin{eqnarray}
\Vert \mathfrak{O}^2 -1\Vert_s^{\gamma, \mathcal{O}_{\infty}^{2\gamma}}\lesssim \left(1+ C\left( \gamma^{-1}\Vert\mathcal{A}-1\Vert^{\gamma, \mathcal{O}_{0}} _{s_0+2\iota+5}\right) \gamma^{-1}\Vert\mathcal{A}-1\Vert^{\gamma, \mathcal{O}_0} _{ s+2\iota+4+s_0}\right).
\end{eqnarray}
\end{lemma}
\begin{proof}By \eqref{def:UV}, \eqref{def:mathcalU} and \eqref{def:mathcalV}, we have
\begin{eqnarray*}
\Omega^2 \circ \Psi (t,x)= -\frac{1}{{\mathcal{A}(\omega t, x )(\partial_{x} \mathcal{U} \left( \omega t, x \right) -1)(\partial_{x} \mathcal{V}\left( \omega t, x \right) +1)}}.
\end{eqnarray*}
Moreover, from Remark \ref{rem:symUV}, we have $\mathcal{V}(\varphi, x) =\mathcal{U}(\varphi, -x)$ and thus, 
\begin{align*}
\mathcal{F}(\varphi, x) &=-\frac{1}{{\mathcal{A}\left(\varphi, x \right)(\partial_{x} \mathcal{U}\left(\varphi, x \right)-1)(\partial_{x} \mathcal{V}\left(\varphi, x \right)+1)}} 
= -\frac{1}{{\mathcal{A}(\varphi,x)(\partial_{x}\mathcal{U}\left(\varphi, x \right)-1)(-\partial_{x}\mathcal{U}\left(\varphi, -x \right)+1)}} \\
&= \frac{1}{{\mathcal{A}(\varphi,x)(1-\partial_{x}\mathcal{U}\left(\varphi, x \right))(1-\partial_{x} \mathcal{U}\left(\varphi, -x \right)}},
\end{align*}
which is clearly even in $x$ since $\mathcal{A}$ is even in $x$. The fact that it is also even with respect to $\varphi$ follows similarly. Moreover, using the algebra estimate \eqref{es:tamealgebra}, item (\ref{item:Ues}) of Proposition \ref{prop:main1} and the usual Neumann series formula\footnote{For $x\in \mathbb{R}$ with $|x|<1$, we have $\frac{1}{1-x}= \sum_{k=0}^{\infty} x^k$.}, we get
\begin{align*}
	\left \Vert \frac{1}{1-\partial_{x}\mathcal{U}}\right \Vert_{s}^{\gamma, \mathcal{O}_\infty^{2\gamma}} &\lesssim \left(1+ C\left( \Vert \mathcal{U} \Vert^{\gamma, \mathcal{O}_\infty^{2\gamma}} _{s_0+1}\right) \right) \gamma^{-1} \Vert \mathcal{A}-1 \vert^{\gamma, \mathcal{O}_0} _{s+2\iota+5} \\
 &\lesssim \left(1+ C\left(\gamma^{-1} \Vert\mathcal{A}-1 \Vert^{\gamma, \mathcal{O}_0} _{s_0+2\iota+5}\right) \right) \gamma^{-1} \Vert \mathcal{A}-1 \Vert^{\gamma, \mathcal{O}_0} _{s+2\iota+5}.
\end{align*}  
Thus, we have 
\begin{align*}
	\Vert \mathcal{F}-1 \Vert_{s}^{\gamma, \mathcal{O}_\infty^{2\gamma}}\lesssim \left(1+ C	\left( \gamma^{-1} \Vert \mathcal{A}-1\Vert^{\mathcal{O}_0 } _{s_0+2\iota+5}\right) \right) \gamma^{-1}\Vert \mathcal{A}-1  \Vert^{\gamma, \mathcal{O}_0} _{s+2\iota+5},
\end{align*} 
where we also used the estimate \eqref{es:tamealgebra} once again. Furthermore, by the estimates on the composition \eqref{es:tamepsi2} with $\Psi^{-1}$, we obtain
\begin{align*} 
\Vert \mathfrak{O}^2 -1\Vert_s^{\gamma, \mathcal{O}_\infty^{2\gamma}} &\lesssim \Vert\mathcal{F}-1\Vert_{s}^{\gamma, \mathcal{O}_\infty^{2\gamma}} + C \left( \gamma^{-1} \Vert \mathcal{A}-1\Vert^{\gamma, \mathcal{O}_0}_{s_0+2\iota+5}  \Vert\mathcal{F}-1\Vert_{s}^{\gamma, \mathcal{O}_\infty^{2\gamma}} +   \gamma^{-1} \Vert \mathcal{A}-1 \Vert^{\gamma, \mathcal{O}_0}_{s+s_0+2\iota+4}  \Vert\mathcal{F} -1\Vert_{2}^{\gamma, \mathcal{O}_\infty^{2\gamma}} \right) \\
&\lesssim  \left(1+ C \left(\gamma^{-1}  \Vert\mathcal{A}-1\Vert^{\gamma, \mathcal{O}_\infty^{2\gamma}} _{s_0+2\iota+5}\right)\gamma^{-1}  \Vert \mathcal{A}-1\Vert^{\gamma, \mathcal{O}_0} _{\max (s+2\iota+5, s+2\iota+s_0+4)}\right) \\
&\lesssim  \left(1+ C\left(\gamma^{-1} \Vert\mathcal{A}-1\Vert^{\gamma, \mathcal{O}_\infty^{2\gamma}} _{s_0+2\iota+5}\right)\gamma^{-1} \Vert \mathcal{A}-1\Vert^{\gamma, \mathcal{O}_0} _{s+2\iota+s_0+4}\right).
\end{align*}
that completes the proof.
\end{proof}

\subsection{The case without parity: proof of Theorem \ref{th:redumetricws}} \label{se:redumetricws}
In this section, we assume that the metric takes the form \eqref{eq:metric} without any additional parity assumptions on $\mathcal{A}$. As before, we would like to construct null coordinates $(u,v)$ verifying the equations \eqref{eq:uv}. To this end, we consider a generalization of the ansatz for $u$ and $v$: 
\begin{equation*} 
U(t,x) = u(t,x)-\frac{t}{\rho_-}+x, \quad V(t,x)= v(t,x)-\frac{t}{\rho_+}-x,
\end{equation*}
where $\rho_\pm$ are adjustable parameters depending only on $\omega$. 

\begin{remark} \label{rem:tcwppr}
 If $\mathcal{A}$ is even in $t$, then it follows from Lemma \ref{lem:parityTR}, that $\rho_-=\rho_+$, so that one can still use the same ansatz as before, cf~Remark \ref{rem:pcwpp}. In order to get something different, one needs $\rho_- \neq \rho_+$ and thus $\mathcal{A}$ without any parity properties. 
\end{remark}
\noindent
As before, we search for quasi-periodic functions $U$ and $V$  and introduce $\mathcal{U}$, $\mathcal{V}$ verifying \eqref{eq:mathcalU}-\eqref{eq:mathcalV} with $\rho= \rho_-$ in the $\mathcal{U}$ equation \eqref{eq:mathcalU} and $\rho=\rho_+$ in the $\mathcal{V}$ equation \eqref{eq:mathcalV}. Proposition \ref{prop:redvf} then guarantees the existence of diffeomorphism $\Psi_\pm : (\varphi, x) \rightarrow (\varphi, y_\pm) $ and Lipschitz functions $\rho_\pm:=\rho_\pm(\omega)$ verifying
\begin{eqnarray*}
\frac{1}{\rho_\pm}:= \left< A \circ \Psi_\pm^{-1} \right>_{\varphi, y}:=\int_{\mathbb{T}^{\nu+1}} ( A \circ \Psi_\pm^{-1}) d\varphi dy ,
\end{eqnarray*}
such that 
\begin{eqnarray*}
\Psi_-(\varphi, x)=(\varphi, x-\mathcal{U}(\varphi, x) ), \quad \Psi_+(\varphi, x)=(\varphi, x+\mathcal{V}(\varphi, x) ) \\ 
\end{eqnarray*}
and $\mathcal{U}$, $\mathcal{V}$ verify \eqref{eq:mathcalU}-\eqref{eq:mathcalV} as well as the estimate \eqref{es:UV}.
Let $\kappa= \rho_+ + \rho_-$ and $\lambda = 2 \rho_+ \rho_-$ and define
\begin{align*}
\tau :&= \frac{\lambda}{\kappa} \frac{v+u}{2}=  t +  \frac{\lambda}{\kappa} \frac{U+V}{2}, \\
R : &= \frac{2}{\kappa} \frac{  \rho_+ v- \rho_- u }{2} = x +  \frac{2}{\kappa} \frac{ \rho_+ V- \rho_- U }{2}. 
\end{align*}
Similarly to Proposition \ref{prop:main2}, we have the following result.
\begin{prop} \label{prop:Cdiffws}
The map 
\begin{align*}
\mathcal{C}: \mathbb{R} \times \mathbb{S}^1 &\rightarrow \mathbb{R} \times \mathbb{S}^1, \quad 
(t,x)  \mapsto (\tau(t,x), R(t,x))	
\end{align*}
is a diffeomorphism. In $(\tau, R)$ coordinates, the metric \eqref{eq:metric} takes the form 
\begin{eqnarray}
g =   \Omega^2 \left( -\frac{1}{\rho_+ \rho_- } d\tau^2 + dR^2 + \frac{\rho_+ -\rho_- }{\rho_- \rho_+}\, 2\, d\tau dR \right),  
\end{eqnarray}
where
\begin{align*}
	\Omega^2 = \frac{1}{ A (1-\partial_{x} U)(1+\partial_{x}V)}.
\end{align*}
\end{prop}
\noindent
As before, given a function 
\begin{align*}
B: \mathbb{R} \times \mathbb{S}^1 &\rightarrow \mathbb{R}, \quad 
(\tau, R)   \rightarrow   B(\tau, R) = \mathcal{B}(\omega  \tau,  R),
\end{align*}
which is quasi-periodic in $\tau$ with frequency vector $\omega$ , we have 
\begin{align*}
B \circ \mathcal{C}(t,x)=  B(\tau(t,x), R(t,x)) 
= \mathcal{B}\left( \omega t + \omega \,  \frac{\lambda}{\kappa} \, \frac{U+V}{2}, x + \frac{2}{\kappa} \frac{ \rho_+ V- \rho_- U }{2} \right), 
\end{align*}
which is quasi-periodic with same frequency vector. Moreover, the coordinate transformation can be lifted to a map defined by 
\begin{align*}
\Psi: \mathbb{T}^{\nu+1} \rightarrow \mathbb{T}^{\nu+1}, \quad 
(\varphi, x) \rightarrow \left( \varphi + \omega\, \frac{\lambda}{\kappa} \, \frac{\mathcal{U}+\mathcal{V}}{2}, x + \frac{2}{\kappa} \frac{ \rho_+ \mathcal{V}- \rho_- \mathcal{U} }{2} \right), 
\end{align*}
The map $\Psi$ is a diffeomorphism and verifies estimates as in Lemmata \ref{lem:psidiff} and \ref{lem:psiest} . Similarly, the conformal factor verifies estimates as in Lemma \ref{lem:confreg}. This completes the proof of Theorem \ref{th:redumetricws}.

\subsection{The case without parity: proof of Corollary \ref{cor:cwp}}

Let $\psi$ verifying the geometric wave equation $\square_g \psi= 0$ with initial data $\psi (t= 0) = f$, $\partial_t \psi (t=0)= g$, with 
\begin{align*}
	(f,g) \in H^{s+1}(\mathbb{S}^1) \times H^{s}(\mathbb{S}^1).
\end{align*}
Let $\mathcal{C}$ be the diffeomorphism of Proposition \ref{prop:Cdiffws}. The function $\psi':= \psi \circ \mathcal{C}^{-1}$ then solves the geometric wave equation $\square_{\eta'} \psi' =0$, where $$\eta':= -\frac{1}{\rho_+ \rho_- } d\tau^2 + dR^2 + \frac{\rho_+ -\rho_- }{\rho_- \rho_+}\, 2\, d\tau dR .$$
First, we have the following result.
\begin{lemma} \label{lem:eneres}
For any $\tau_0 \in \mathbb{R}$, the map $(f,g) \rightarrow (\psi'(\tau_0), \partial_\tau \psi'(\tau_0) )$, where $\psi'= \psi \circ \mathcal{C}^{-1}$ is the solution to the initial value problem in $(\tau,R)$ coordinates as above, is an isomorphism of $H^{s+1}(\mathbb{S}^1) \times H^{s}(\mathbb{S}^1)$. 
\end{lemma}
\begin{proof}
This follows from the fact that the hypersurface $\Sigma_{\tau_0} := \left\{ (t,x)\,\, /\,\,\tau(t,x)= \tau_0 \right\}$ is timelike, the well-posedness of the wave equation and standard (local in time) energy estimates for the wave operator. 
\end{proof}
\noindent
In $(\tau, R)$ coordinates, the wave equation can be solved explicitly. 
\begin{lemma}
Let $\psi' $ be a solution to the wave equation  $\square_{\eta'} \psi' =0$ with data $(f', g') \in H^{s+1}(\mathbb{S}^1) \times H^{s}(\mathbb{S}^1)$ on $\{\tau=0\}$. Then, $\psi'$ admits the representation 
\begin{eqnarray*}
\psi'(\tau, R) = \sum_{k \in \mathbb{Z}} \sum_{\pm} \psi'_{\pm, k}e^{i \omega_\pm(k) \tau} e^{i k R}, 
\end{eqnarray*}
where $\omega_{\pm}(k) \sim \frac{k}{\sqrt{\rho_+ \rho_-}}$, as $|k| \rightarrow +\infty$, are the solutions to 
\begin{align*}
	\rho_+ \rho_- \omega^2 - 2 \left( \rho_+ - \rho_-\right) k w -k^2=0
\end{align*} 
and we have the bound
\begin{align*}
	\sum_{\pm} \sum_{k} (1+k^2)^{s/2} \vert  \psi'_{\pm,k} \vert^2 \lesssim \Vert (f', g') \Vert_{H^{s+1}(\mathbb{S}^1) \times H^{s}(\mathbb{S}^1)}.
\end{align*}
\end{lemma}
\begin{proof}
The components of the inverse metric are given by 
\begin{eqnarray*}
(\eta')^{-1}_{\alpha \beta} = \frac{\rho_+\rho_-}{\rho_+^2+ \rho_-^2- \rho_+ \rho_-}\left(\begin{array}{cc}
- \rho_+ \rho_- & \rho_+ - \rho_- \\
\rho_+ -\rho_- & 1 
\end{array} \right)
\end{eqnarray*}
and thus $\psi'$ solves
\begin{align*}
	- \rho_+ \rho_- \partial_{\tau}^2 \psi' + 2 \partial_\tau \partial_R \psi' + \partial_R^2 \psi'=0.
\end{align*}
The lemma then follows by decomposing $\psi'$ on the modes $e^{ikR}$ and standard Fourier series analysis. 
\end{proof}
\noindent
We can now complete the proof of Corollary \ref{cor:cwp}. Let $(f,g) \in H^{s+1}(\mathbb{S}^1) \times H^{s}(\mathbb{S}^1)$ and let $\psi'= \psi \circ \mathcal{C}^{-1} $ be as in Lemma \ref{lem:eneres}. Fix a hypersurface $\Sigma_{\tau_0}=\{\tau=\tau_0\}$. By Lemma \ref{lem:eneres}, we have
\begin{align*}
	\Vert \psi'(\tau_0) \Vert_{s+1}+ \Vert \partial_\tau \psi'(\tau_0) \Vert_s \lesssim \Vert f \Vert_{s+1} + \Vert g \Vert_s. 
\end{align*}
Since $\psi'$ verifies an equation with constant coefficients, $\partial_\tau^\ell \partial_R^m \psi$ verifies the same equation for any $\ell +m \le s-1$. In particular, conservation of energy between the slices $\tau=\tau_0$ and $\tau=t \in \mathbb{R}$ gives that 
\begin{align*}
	\sum_{\ell+m \leq s-1} \int_{  A_{t}  } \left( | \partial_\tau^\ell \partial_{R+1}^m \psi' |^2 + |\partial_\tau^{\ell+1} \partial_R^m \psi' |^2 \right) d  A_{t}  \lesssim  \| \psi'(\tau_0) \|_{s+1}+  \| \partial_\tau \psi'(\tau_0) \|_s \lesssim  \| f \| _{s+1} +  \| g  \| _s, 
\end{align*}
where $ d A_{t}  $ is the volume form on $  A_{t}   $ for the metric $\eta'$, which is uniformly bounded from above and below. Since $\partial_t$ and $\partial_x$ can be rewritten in terms of $\partial_\tau$ and $\partial_R$ with coefficients depending only on $\mathcal{U}$ and $\mathcal{V}$, which are uniformly bounded together with their derivatives, the corollary follows since the left hand side is comparable with the $H^{s+1} \times H^s$ norm of $\psi$. 

\section{Reducibility of the reversible Klein-Gordon equation} \label{se:rkg}



\subsection{The Klein-Gordon equation in $(\tau,R)$ coordinates} \label{se:kgtR}
In this section, we prove Corollary \ref{cor:kgred}. Consider thus a Klein-Gordon equation of the form \eqref{eq:KG} with coefficients verifying the algebraic assumptions $a,b,c$ as in Corollary \ref{cor:kgred} as well as the smallness assumption \eqref{ineq:Asmallness} with $s_1 \ge s_0 + 2 \iota +5$ and $\mathcal{A}^2= 1- \mathcal{B}_{xx}$ which verifies the algebraic conditions of Theorem \ref{th:redumetric}, i.e.~$\mathcal{A}$ is even in $t$ and $x$ independently.  \\ \\
We start by applying Theorem \ref{th:redumetric} and obtain $\alpha(\omega)$, the diffeomophism $\mathcal{C}_\omega=(\tau,R)$, the conformal factor $\Omega$, the frequency set $\mathcal{O}^{2\gamma}_\infty$,  as well as the lifted diffeomorphism $\Psi$ as in \eqref{def:Psi}. Corollary \ref{cor:kgred} then follows directly from the following lemma. 

\begin{lemma} \label{lem:kgTR}
For any frequency $\omega \in \mathcal{O}^{2\gamma}_\infty$, if $\psi$ solves $\eqref{eq:KG}$, then $\phi:= \psi \circ \mathcal{C}_\omega^{-1}$ verifies 
\begin{align} 
& \left(- \alpha^2(\omega) \partial_{\tau}^2  + \partial_{R}^2  - \mathtt{m}\right) \phi \nonumber \\
 &= \Omega^2 \left[ \left( \frac{B^x+ A  \partial_{x} A }{A} \right) \partial_{x} \psi  + \left( \frac{B^t}{A} + \frac{ \partial_{t} A }{A^2}\right) \partial_{t} \psi  + \frac{B}{A}\psi- \mathtt{m}\left( 1- \frac{1}{A} \right) \psi  \right] \circ \mathcal{C}^{-1} 
   \hbox{} - \mathtt{m}\left( 1- \Omega^2 \right) \phi, \label{eq:wtauR1}\\
&= G^R \partial_{R}\phi + G^\tau \partial_{\tau }\phi +G \phi, \label{eq:wtauR2}
\end{align}
where 
\begin{align*}
	G^R=\mathcal{G}^R(\omega \tau,R), \quad  G^\tau=\mathcal{G}^\tau(\omega \tau,R), \quad G=\mathcal{G}(\omega \tau,R), 
\end{align*}
for some functions $\mathcal{G}^R$, $\mathcal{G}^\tau$, $\mathcal{G}: \mathbb{T}^\nu_\varphi \times \mathbb{T}_R \rightarrow \mathbb{R}$ verifying the algebraic properties 
\begin{itemize}[leftmargin=5.5mm]
\item[a.] $\mathcal{G}$ is even in $\varphi$ and $R$ independently, \label{item:G}
\item[b.] $\mathcal{G}^R$ is even in $\varphi$ and odd in $R$,  \label{item:GR}
\item[c.] $\mathcal{G}^\tau$ is odd in $\varphi$ and even in $R$.  \label{item:GT}
\end{itemize}
as well the estimates
\begin{eqnarray} \label{es:mathcalg}
\left|\left|\mathcal{G}^R  \right| \right|_s+
\left|\left|  \mathcal{G}^\tau \right| \right|_s+
\left|\left| \mathcal{G} \right| \right|_s \lesssim \max_{a \in \left\{  \mathcal{B}^{xx}, \mathcal{B}^{x}, \mathcal{B}^{t}, \mathcal{B} \right\}} \gamma^{-1}  \|    a    \|_{H^{s+s_0+2\iota+5}(\mathbb{T}^{\nu+1})}
\end{eqnarray}

\end{lemma}
\begin{proof}
From \eqref{eq:KGgeom} and \eqref{eq:gmetrictauR}, we have that \eqref{eq:wtauR1} holds. By the chain rule, we also have 
\begin{align*}
	\left[ \left( \frac{B^x+ A  \partial_{x} A }{A} \right) \partial_{x} \psi  + \left( \frac{B^t}{A} + \frac{ \partial_{t} A }{A^2}\right) \partial_{x} \psi  \right] \circ \mathcal{C}^{-1} 
	&=  \left( \frac{B^x+ A \partial_{x} A }{A} \right) \circ \mathcal{C}^{-1} \cdot \left( \partial_{\tau} \phi  \left(\partial_{x}\tau \circ \mathcal{C}^{-1}\right)  + \partial_{R} \phi  \left(\partial_{x} R \circ \mathcal{C}^{-1}\right)  \right) \\
&  \hbox{}+  \left( \frac{B^t}{A} + \frac{ \partial_{t} A }{A^2}\right) \circ \mathcal{C}^{-1} \cdot \left( \partial_{\tau} \phi  \left(\partial_{t}\tau \circ \mathcal{C}^{-1}\right)  + \partial_{R} \phi  \left(\partial_{t} R \circ \mathcal{C}^{-1}\right)  \right) 
\end{align*} 
Define thus 
\begin{align*}
 \mathcal{G}^R(\omega \tau,R)&:= \Omega^{2} \left(\left( \frac{B^x+ A \partial_{x} A }{A} \right) \partial_{x} R \right)   \circ \mathcal{C}^{-1}+ \Omega^{2} \left( \left( \frac{B^t}{A} + \frac{ \partial_{t} A }{A^2}\right)\partial_{t}R \right)\circ \mathcal{C}^{-1},  \\
\mathcal{G}^\tau(\omega \tau,R)&:=  \Omega^{2} \left(\left( \frac{B^x+ A \partial_{x} A }{A} \right)\partial_{x}\tau \right)   \circ \mathcal{C}^{-1}+   \Omega^{2} \left( \left( \frac{B^t}{A} + \frac{ \partial_{t} A }{A^2}\right)\partial_{t}\tau \right)\circ \mathcal{C}^{-1}  ,\\
\mathcal{G}(\omega \tau,R)&:= \Omega^{2}\left( \frac{B}{A}- \mathtt{m}\left( 1- \frac{1}{A} \right)\right)  \circ \mathcal{C}^{-1}- \mathtt{m}\left( 1- \Omega^2 \right).
\end{align*}
Alternatively, one can defined directly the functions $\mathcal{G}^R$, $\mathcal{G}^\tau$, $\mathcal{G}$ using the $\Psi$ diffeomorphism (instead of $\mathcal{C}$), so that for instance 
\begin{align*}
	\mathcal{G}=  \mathfrak{O}^{2}\left(\frac{\mathcal{B}}{\mathcal{A}} -\mathtt{m}\left( 1- \frac{1}{\mathcal{A}} \right)\right)  \circ \Psi^{-1}- \mathtt{m}\left( 1- \mathfrak{O}^2 \right).
\end{align*}
Then, the fact that the functions $\mathcal{G}^R$, $\mathcal{G}^\tau$ and $\mathcal{G}$  verify their required algebraic properties, follows  due to $\mathcal{C}$ preserves evenness or oddness with respect to $x$ and $t$ independently, together with the individual algebraic properties of $R$, $\tau$, the various $B$ coefficients and $A$. For instance, since $B^x$ and $ \partial_{x} A $ are odd is $x$ and since $A$ and $\partial_{x}R$ are even in $x$ (since $R$ is odd in $x$), it follows that  $\frac{B^x+ A \partial_{x} A }{A}\partial_{x}R$ is odd in $x$ and since $\mathcal{C}^{-1}$ preserves oddness in $R$, we have that $\left( \frac{B^x+ A \partial_{x} A }{A} \partial_{x}R \right)   \circ \mathcal{C}^{-1}$ is odd in $R$. 
Using Lemma \ref{lem:psiest} and the estimate \eqref{es:confor}, we then have
\begin{align*}
\Vert \mathcal{G} \Vert_s &\lesssim  \Vert\mathfrak{O}^{2}\frac{\mathcal{B}}{\mathcal{A} }\circ \Psi^{-1}\Vert_s + \left \Vert \mathfrak{O}^{2} \mathtt{m}\left( 1- \frac{1}{\mathcal{A}} \right) \circ \Psi^{-1} \right \Vert_s + \gamma^{-1} \Vert\mathcal{A}-1\Vert^{\gamma, \mathcal{O}_0} _{ s+s_0+2\iota+4} \\
&\lesssim \left \Vert \mathfrak{O}^{2} \right\Vert_{s_0} \left \Vert \frac{\mathcal{B}}{\mathcal{A} }\circ \Psi^{-1}  \right \Vert_s+ \left \Vert \mathfrak{O}^{2} \right\Vert_{s} \left \Vert  \frac{\mathcal{B}}{\mathcal{A} }\circ \Psi^{-1}  \right \Vert_{s_0} \\
& \quad \quad  +
\left \Vert \mathfrak{O}^{2} \right\Vert_{s_0} \left \Vert \mathtt{m}\left( 1- \frac{1}{\mathcal{A}} \right) \circ \Psi^{-1} \right \Vert_s+ \left \Vert \mathfrak{O}^{2} \right\Vert_{s} \left \Vert \mathtt{m}\left( 1- \frac{1}{\mathcal{A}} \right) \circ \Psi^{-1} \right \Vert_{s_0}
+  \gamma^{-1} \Vert\mathcal{A}-1\Vert^{\gamma, \mathcal{O}_0} _{ s+s_0+2\iota+4} 
\end{align*}
as well as 
\begin{align*}
 \left \Vert \mathtt{m}\left( 1- \frac{1}{\mathcal{A}} \right) \circ \Psi^{-1} \right \Vert_{s}
&\lesssim   \left \Vert 1- \frac{1}{\mathcal{A} }\right \Vert_s + C \left( \gamma^{-1} \Vert\mathcal{A}-1\Vert_{s_0+2\iota+5}  \Vert \mathcal{A} \Vert_s+  \gamma^{-1} \Vert \mathcal{A}-1 \Vert _{s+s_0+2\iota+4} \left \Vert 1- \frac{1}{\mathcal{A}} \right \Vert_1 \right)
\end{align*}
and similarly for $\frac{\mathcal{B}}{\mathcal{A} }\circ \Psi^{-1}$.
Since $\mathcal{A}= \sqrt{1-\mathcal{B}^{xx}}$, using the smallness condition \eqref{ineq:Asmallness} and standard tame estimates for compositions, we obtain 
\begin{align*}
\Vert \mathcal{G} \Vert_s
 &\lesssim  \gamma^{-1}  \Vert \mathcal{B}^{xx} \Vert_{s+s_0+2\iota+4} + C \left( \gamma^{-1} \Vert \mathcal{B}^{xx}\Vert_{s_0+2\iota+5}  \left(\Vert \mathcal{B}^{xx} \Vert_s+ \Vert \mathcal{B} \Vert_s\right)+  \gamma^{-1} \Vert \mathcal{B}^{xx}\Vert_{s+s_0+2\tau+4} \left( \Vert \mathcal{B}^{xx} \Vert_1 +\Vert \mathcal{B}\Vert_1 \right) \right) \\
 &\lesssim   \gamma^{-1} \Vert \mathcal{B}^{xx}\Vert_{s+s_0+2\iota+4}.
\end{align*}
In order to estimate $\partial R \circ \mathcal{C}^{-1}$, $\partial \tau \circ \mathcal{C}^{-1}$, $\mathcal{G}^\tau$, $\mathcal{G}^R$, we proceed similarly, using the smallness \eqref{ineq:Asmallness}, standard tame estimates for compositions and Lemma \ref{lem:psiest}. For instance, 
\begin{align*}
\Vert \partial_{x} R  \circ \mathcal{C}^{-1} \Vert_s &\lesssim \Vert \partial_{x} R \Vert_s + C \left( \gamma^{-1} \Vert \mathcal{B}^{xx} \Vert_{s_0+2\iota+5}  \Vert \partial_{x}R \Vert_s+  \gamma^{-1} \Vert \mathcal{B}^{xx} \Vert_{s+s_0+2\tau+4} \Vert \partial_{x}R \Vert_1 \right) \\
&\lesssim 
\Vert \partial_{x}\mathcal{U}  \Vert_s +
\Vert   \partial_{x}\mathcal{V} \Vert_s + C  \Big[ \gamma^{-1} \Vert \mathcal{B}^{xx}\Vert_{s_0+2\iota+5}  
\left(
\Vert \partial_{x} \mathcal{V}  \Vert_s+
\Vert   \partial_{x}\mathcal{U} \Vert_s
\right) \\
&+  \gamma^{-1} \Vert \mathcal{B}^{xx} \Vert_{s+s_0+2\iota+4} 
\left(
\Vert \partial_{x} \mathcal{V}  \Vert_1+
\Vert   \partial_{x}\mathcal{U} \Vert_1
\right) \Big] \\
&\lesssim\gamma^{-1}  \Vert \mathcal{B}^{xx} \Vert_{s+2\iota+5} + C \left( \gamma^{-1} \Vert \mathcal{B}^{xx} \Vert_{s_0+2\iota+5}  \Vert  \mathcal{B}^{xx} \Vert_{s+2\iota+5}+  \gamma^{-1} \Vert \mathcal{B}^{xx} \Vert_{s+s_0+2\iota+4} \Vert  \mathcal{B}^{xx} \Vert_1 \right) \\
&\lesssim\gamma^{-1}  \Vert \mathcal{B}^{xx} \Vert_{s+s_0+2\iota+4},
%
\end{align*}
It then easily follows that $\mathcal{G}^R$ and $\mathcal{G}^\tau$ verify the estimate \eqref{es:mathcalg} that completes the proof.
\end{proof}

\subsection{Reducibility of the first order terms} \label{se:redfirsto}
In view of the previous lemma, we have now reduced the problem to that of studying 
\begin{align}
-\alpha^2(\omega) \partial_{\tau}^2  \phi + \partial_{R}^2 \phi  - \mathtt{m}\phi = G^R \partial_{R}\phi + G^\tau  \partial_{\tau}\phi +G \phi, \label{eq:wtauR3} 
\end{align}
where the functions $G$, $G^R$ and $G^\tau$ verify the algebraic properties $(a)$, $(b)$ and $(c)$ as well as the estimates \eqref{es:mathcalg} of Lemma \ref{lem:kgTR} where $\omega \in \mathcal{O}^{2\gamma}_\infty$ satisfies the diophantine properties introduced in \eqref{eq:diophw}. From now on, and until the end of this section, we shall focus on this equation. \\ \\
A priori, this is a setting where we can apply the techniques developed in \cite{berti2024reducibility} for the non-maximal terms in the Klein-Gordon equation with the following differences: 
\begin{itemize}[leftmargin=5.5mm]
\item Equation \eqref{eq:wtauR3} only makes sense for $\omega \in \mathcal{O}^{2\gamma}_\infty$. In \cite{berti2024reducibility}, the initial equation is defined for all $\omega \in \Lambda=[-1/2, 1/2]^{\nu}$ and is restricted to $\omega \in \mathcal{O}^{2\gamma}_\infty$ only in Section 9, that is after the introduction of complex variables and the reformation of the equation as a first order system as well as after the symmetrizations of Section 8.
\item The wave operator in \eqref{eq:wtauR3} has a dependence through $\omega$ via $\alpha(\omega)$. This is akin to the $\mathfrak{c}$ term appearing say in Proposition 9.2 of \cite{berti2024reducibility}. This has a very minor role, in fact, we could at this stage, redefine $\tau$ to absorb $\alpha$, at the price of also changing $\omega$. 
\item Equation \eqref{eq:wtauR3} has an extra term $G^\tau \partial_{\tau} \phi $ compared to the Klein-Gordon equation studied in \cite{berti2024reducibility}. 
\end{itemize}
These are all minor modifications from the setting of \cite{berti2024reducibility} and thus a similar analysis could be carried out. However, through propositions \ref{prop:redutime1} and \ref{prop:reducall1}, we will revisit the reduction of the terms of order $1$. Our aim is to emphasize that our treatment of the maximal terms has some benefits for the lower order terms as well, essentially because, at this stage of the reduction, we generated only a few error terms and in particular those that are still differential operators of order $1$. 

\subsection{Proof of Proposition \ref{prop:redutime1}}
First, we prove Proposition \ref{prop:redutime1}, which allows us to remove the   time-derivative via a multiplication operator. 
\begin{proof}Let $\mathcal{P} : = e^{-\frac{1}{2\alpha^2}\mathcal{H}(\varphi, R)}$, where $\mathcal{H}$ solves 
\begin{align*}
	\omega\cdot \partial_{\varphi} \mathcal{H} = \mathcal{G}^\tau, \quad \langle \mathcal{P} \rangle_{\varphi, R}= 0, 
\end{align*} 
which makes sense since $\mathcal{G}^\tau$ is odd with respect to  $\tau$ and $\langle \mathcal{G}^\tau \rangle=0$. Note that, for $\omega \in \mathcal{O}^{2\gamma}_\infty$, we have 
\begin{align*}
	\Vert\mathcal{H} \Vert_s \lesssim \gamma^{-1} \Vert \mathcal{G}^\tau \Vert_{s+\iota}
\end{align*}
and hence 
\begin{align}\label{es:P}
\Vert\mathcal{P}-1 \Vert_s \lesssim \gamma^{-1} \Vert \mathcal{G}^\tau \Vert_{s+\iota}.
\end{align}
Let $P:=\mathcal{P}(\omega t, R)$, which then verifies $\partial_\tau P = - \frac 1 2 \alpha^{-2} G^\tau P$.
Let $\mathcal{L}_1$ and $\mathcal{L}_2$ be as in Proposition \ref{prop:redutime1} and compute 
\begin{align*}
& \mathcal{L}_2 \phi=-\alpha^2(\omega) \partial_{\tau}^2  (P\phi) + \partial_{R}^2 (P\phi)  - \mathtt{m}(P \phi) - G^R \partial_{R} (P\phi) = \\
& P\left( - \alpha^2(\omega) \partial_{\tau}^2   + \partial_{R}^2   - \mathtt{m}- G^R \partial_R\right) \phi +  \phi \left( -\alpha^2(\omega) \partial_{\tau}^2   + \partial_{R}^2  - G^R \partial_R \right) P 
 - 2\alpha^2 \partial_\tau P \partial_\tau \phi + 2\partial_R P \partial_R \phi =\\
& P \mathcal{L}_1 \phi - G^R_2 \partial_R (P \phi)+ G^P \phi , 
\end{align*}
where 
\begin{align*}
G^P :=  -2  \frac{\left(\partial_R P \right)^2}{P}  +  \left( -\alpha^2(\omega) \partial_{\tau}^2   + \partial_{R}^2  - G^R \partial_R \right) P ,  \quad 
G_2^R:= - 2 P^{-1} \partial_R P.
\end{align*}
The estimates \eqref{es:mathcalgp} then follow directly from the definitions and \eqref{es:P}. 
\end{proof}

\begin{remark}
In particular, we see from the above that if, for some quasi-periodic function $F$,  $G^\tau =  \partial_\tau F$, $G^R=- \alpha^{-2}\partial_R F$ in \eqref{eq:wtauR2}, then $G^\tau \partial_{\tau}\phi + G^R \partial_{R}\phi$ is a \emph{null form} which can be replaced by a zeroth order term using the $P$ multiplication operator introduced above. 
\end{remark}
\begin{remark}
Note that the operator $\phi \mapsto P\phi$ is of course a tame operator in $H^s$, due to the tame product estimate \eqref{es:tamealgebra}. 
\end{remark}
\subsection{Proof of Proposition \ref{prop:reducall1}} \label{se:proopr:reducall1}
From now on, we simplify the notation by writing $\phi$ instead of $P\phi$, where $P$ is the multiplication operator defined above. At this stage,  we have reduced the problem to that of studying the equation 
\begin{align}
-\alpha^2(\omega) \partial_{\tau}^2  \phi + \partial_{R}^2 \phi  - \mathtt{m}\phi = G^R \partial_{R} \phi  + G \phi. \label{eq:wtauRP2}
\end{align}
We shall first prove below that the first order error term on the right-hand side can be eliminated by physical space transformations and classical differential operators. However, the transformations that we use in this first attempt below do not a priori preserve the parity properties of the equation. We then revisit this transformation replacing differential operators in $R$ by pseudo-differential operators in order to respect this parity. We prefer to still present the physical space transformation below because, in our opinion, it helps the reader understand how the transformations can be explicitly constructed.\\ \\
Let us thus start with the physical space analysis. Let $\phi$ solves \eqref{eq:wtauRP2} and consider the null derivatives
\begin{align*}
L := \alpha \partial_\tau + \partial_R, \quad 
\underline{L}:= \alpha \partial_\tau - \partial_R,
\end{align*}
so that  \eqref{eq:wtauRP2} becomes 
\begin{align*}
- L \underline{L} \phi -  \mathtt{m}\phi = \frac{G^R}{2} \left( L - \underline{L} \right) \phi  +G \phi =  \frac{G^R}{2} L \phi - \frac{G^R}{2}  \underline{L} \phi + G \phi.
\end{align*}
Up to a zeroth order term, the first term on the right-hand-side can be absorbed on the left-hand-side, that os
\begin{align}
- L \left( \underline{L} \phi+\frac{G^R}{2} \phi \right)-  \mathtt{m}\phi &= - \frac{G^R}{2}  \underline{L} \phi + \left( G -\frac{L(G^R)}{2} \right) \phi  \nonumber\\
&= - \frac{G^R}{2} \left(  \underline{L} \phi + \frac{G^R}{2} \phi \right)  + \left( G -\frac{L(G^R)}{2} + \frac{(G^R)^2}{4} \right) \phi. \label{eq:LbL2}
\end{align}
The second term on the right-hand-side is a 0th order term, but it remains to remove the first term on the right-hand-side which still contains derivatives of $\phi$. To this end, consider a transport equation of the form
\begin{equation} \label{eq:transH}
L (H) = \frac{G^R}{2}.
\end{equation}
Note that, in view of the parity condition on $G^R$, the above equation is solvable among quasi-periodic functions $H(\tau, R):= \mathcal{H}(\omega t, R)$ by considering the equation
\begin{align*}
	\alpha\, \omega. \partial_\varphi  \mathcal{H} + \partial_R \mathcal{H} = \frac{\mathcal{G}^R}{2}.
\end{align*}
Moreover, we can estimate the norm of $\mathcal{H}$ by using the diophantine condition $\omega \in \mathcal{O}^{2\gamma}_{\infty}$. Assuming we have solved the transport equation \eqref{eq:transH} for $H$, we can rewrite \eqref{eq:LbL2} as
\begin{eqnarray*}
- L \left( e^{-H} \left(\underline{L}  +\frac{G^R}{2} \right) \phi \right)-  \mathtt{m}e^{-H} \phi = e^{-H} \left( G -\frac{L(G^R)}{2} + \frac{(G^R)^2}{4} \right) \phi
\end{eqnarray*}
Replaceing $L$ by $\underline{L}$, we also have 
\begin{eqnarray*}
- \underline{L} \left( e^{-\underline{H}} \left(L -\frac{G^R}{2} \right) \phi \right)-  \mathtt{m}e^{-\underline{H}} \phi = e^{-\underline{H}} \left( G +\frac{L(G^R)}{2} + \frac{(G^R)^2}{4} \right) \phi, 
\end{eqnarray*}
where  $\underline{H}$ solves  
\begin{align*}
	\underline{L}\, \underline{H} = - \frac{G^R}{2}.
\end{align*}
Recall that the second order wave operator $- L \underline{L}$ can be rewritten as a first order operator using $(L \phi, \underline{L}\phi)$ as the new unknown. The above shows that the quasi-periodic transformation\footnote{Note that for the map to be well-defined, one needs to be able to reconstruct $\phi$ from $L\phi$ and $\underline{L}\phi$, which is actually only true modulo constants.}
\begin{align*}
	(L \phi , \underline{L} \phi) \rightarrow \left(e^{-\underline{H}} \left(L   -\frac{G^R}{2} \right) \phi , e^{-H} \left(\underline{L} +\frac{G^R}{2} \right) \phi \right)
\end{align*}
allows us to remove all terms of order $1$. This would then leave only terms of order $0$. These however cannot be treated by KAM schemes a priori because the worst coefficient is of size $ \mathtt{m}$ and therefore not small. Moreover, the above transformation does not preserve the parity properties of the equations. In fact, this is already true for the operators $L$ and $\underline{L}$ due to the presence of the $\partial_R$ derivatives. Working instead, as in \cite{berti2024reducibility}, with the pseudo-differential analogues
 \begin{align*}
 K :=-  i \alpha \partial_\tau + D_{\mathrm{m}} ,\quad 
 \underline{K} &:= i \alpha \partial_\tau + D_{\mathrm{m}}, 
 \end{align*}
where $D_{\mathrm{m}} := \mathrm{Op}( \sqrt{|\xi|^2+\mathrm{m}})$, we can follow the same strategy to remove the terms of order $1$ as we show below. The advantage is that the worst 0th order coefficient will now be small and the transformation will preserve parity. \\ \\
Define the complex variables\footnote{We follow the notations and normalization of \cite{berti2024reducibility} to ease the comparison.}
\begin{align*}
v: = \partial_{\tau} \phi, \quad 
\left[ \begin{array}{c} u \\ \underline{u} \end{array} \right]:&= \mathcal{C} \left[ \begin{array}{c} \phi \\ v \end{array} \right], \quad 
\mathcal{C}: = \frac{1}{\sqrt{2}} \left( \begin{array}{cc} D_m & - \alpha i \\ D_m &\alpha i \end{array} \right) 
\end{align*}
so that
\begin{align*}
	\left[ \begin{array}{c} u \\ \underline{u} \end{array} \right]&=  \frac{1}{\sqrt{2}} \left[ \begin{array}{c} D_m\phi - i \alpha v\\ D_m \phi + i \alpha v \end{array} \right]=  \frac{1}{\sqrt{2}} \left[ \begin{array}{c} K\phi\\ \underline{K} \phi  \end{array} \right].
\end{align*}
Note that
\begin{align*}
	\mathcal{C}^{-1}:&= \frac{1}{\sqrt{2}} \left( \begin{array}{cc} D_m^{-1}  & D_m^{-1} \\  \frac{i}{\alpha}    &-\frac{i}{\alpha } \end{array} \right).
\end{align*}
We now turn to the proof of Proposition \ref{prop:reducall1}. 
\begin{proof}
The wave equation \eqref{eq:wtauRP2} can be rewritten as 
\begin{eqnarray*}
- K \underline{K} \phi &=& G^R  \partial_{R}\phi + G \phi,
\end{eqnarray*}
or equivalently, in first order form, using $\phi= \frac{1}{2}D_{\mathrm{m}}^{-1} \left(K \phi + \underline{K}\phi \right)= \frac{1}{\sqrt{2}}D_{\mathrm{m}}^{-1} \left(u + \underline{u} \right) $, as  
\begin{align*}
\underline{K} u &= -\frac{1}{\sqrt{2}}\left( G^R \partial_{R}\phi + G \phi  \right)  \\
&= -\frac{1}{\sqrt{2}}\left( G^R  \mathrm{Op} (i \xi) D_{\mathrm{m}}^{-1}\frac{1}{2}\left( K + \underline{K}\right)\phi + G D_\mathrm{m}^{-1}\frac{1}{2}\left( K + \underline{K}\right) \phi  \right) \\
&= -\frac{1}{2\sqrt{2}}\left( G^R   \mathrm{Op} (i \xi) D_{\mathrm{m}}^{-1} K \phi +  G^R   \mathrm{Op} (i \xi) D_{\mathrm{m}}^{-1}\underline{K}\phi+ G D_\mathrm{m}^{-1}K(\phi) + G D_\mathrm{m}^{-1}\underline{K}(\phi) \right) \\
&= -\frac{1}{2\sqrt{2}}\underline{K} \left( G^R  \mathrm{Op} (i \xi) D_{\mathrm{m}}^{-1} \phi \right)   -\frac{1}{2\sqrt{2}} \left[ G^R  \mathrm{Op} (i \xi) D_{\mathrm{m}}^{-1}, \underline{K}\right] \phi \\
&\quad  -\frac{1}{2\sqrt{2}} \left( G^R  \mathrm{Op} (i \xi) D_{\mathrm{m}}^{-1} K\phi+ G D_\mathrm{m}^{-1}K(\phi) + G D_\mathrm{m}^{-1}\underline{K}(\phi) \right)  
\end{align*}
which can be rewritten as 
\begin{align*}
 \underline{K} \left( u + \frac{1}{2\sqrt{2}} \left( G^R \mathrm{Op}(i \xi) D_\mathrm{m}^{-1} \phi \right) \right) &=
 -\frac{1}{2}  G^R  \mathrm{Op} (i \xi) D_{\mathrm{m}}^{-1} u  
  -\frac{1}{2} \left(  G D_\mathrm{m}^{-1}u  + G D_\mathrm{m}^{-1}\underline{u} \right)  \\
 & -\frac{1}{2} \left[ G^R  \mathrm{Op} (i \xi) D_{\mathrm{m}}^{-1}, \underline{K}\right] D_\mathrm{m}^{-1} \frac{u + \underline{u}}{2}. 
\end{align*}
Similarly, we compute
\begin{align*}
K \left( \underline{u} + \frac{1}{2\sqrt{2}} \left( G^R \mathrm{Op}(i \xi) D_\mathrm{m}^{-1} \phi \right) \right) &=
-\frac{1}{2}  G^R  \mathrm{Op} (i \xi) D_{\mathrm{m}}^{-1} \underline{u} 
 -\frac{1}{2} \left(  G D_\mathrm{m}^{-1}u  + G D_\mathrm{m}^{-1}\underline{u} \right) \\
&-\frac{1}{2} \left[ G^R  \mathrm{Op} (i \xi) D_{\mathrm{m}}^{-1}, K\right] D_\mathrm{m}^{-1} \frac{u + \underline{u}}{2}. 
\end{align*}
Consider the transformation 
\begin{eqnarray*}
\mathfrak{V}: H^s(\mathbb{T}; \mathbb{C}^2) &\rightarrow& H^s(\mathbb{T}; \mathbb{C}^2) \\
(u,\underline{u}) &\mapsto& \mathfrak{V}(u,\underline{u}):= (v(u, \underline{u}), \underline{v}(u, \underline{u}))
\end{eqnarray*}
where 
\begin{align*}
v(u, \underline{u})&:=  u + \frac{1}{2\sqrt{2}} \left( G^R \mathrm{Op}(i \xi) D_\mathrm{m}^{-1} \frac{1}{\sqrt{2}}D_{\mathrm{m}}^{-1} \left(u + \underline{u} \right)  \right), \\
 \underline{v}(u, \underline{u})&:=  \underline{u} + \frac{1}{2\sqrt{2}} \left( G^R \mathrm{Op}(i \xi) D_\mathrm{m}^{-1}\frac{1}{\sqrt{2}}D_{\mathrm{m}}^{-1} \left(u + \underline{u} \right) \right).
\end{align*}
We have that
\begin{align*}
	R:=\mathfrak{V}-\mathrm{Id}  
\end{align*}
is a matrix of pseudo-differential operators of order $-1$ and one can readily verify that 
\begin{itemize}
\item $R$ is real-to-real, 
\item $R$ is parity-preserving, 
\item $R$ is reversibility-preserving. 
\end{itemize}
 Moreover, using Lemmata \ref{sobaction} and \ref{lem:pseudocom}, we have, for all $s \ge s_0$, for all $p \ge 0$, 
\begin{align}\label{EstimatesR}
\begin{split}
	\Vert R \Vert^{\gamma, \mathcal{O}^{2\gamma}_\infty}_{-1,s,p} &\lesssim  \Vert \mathcal{G}^R \Vert^{\gamma, \mathcal{O}^{2\gamma}_\infty}_s \\
\Vert R (u, \underline{u}) \Vert^{\gamma, \mathcal{O}^{2\gamma}_\infty}_s &\lesssim   \Vert \mathcal{G}^R \Vert^{\gamma, \mathcal{O}^{2\gamma}_\infty}_s \Vert (u, \underline{u})  \Vert^{\gamma, \mathcal{O}^{2\gamma}_\infty}_{s_0-1}+ \Vert \mathcal{G}^R \Vert^{\gamma, \mathcal{O}^{2\gamma}_\infty}_{s_0} \Vert (u, \underline{u})  \Vert^{\gamma, \mathcal{O}^{2\gamma}_\infty}_{s-1}.
\end{split}
\end{align}
By Neumann series, using the smallness assumption \eqref{ass:smallnessGR} on $\mathcal{G}^R$, one has that $\mathfrak{V}$ is invertible  and $ \mathfrak{V}^{-1}-\mathrm{Id}$ verifies the same estimates as $R$. We will write $\mathfrak{V}^{-1}(v, \underline{v})= \left(u (v,\underline{v}) , \underline{u}(v, \underline{v}) \right)$ and we drop the dependence on $(v,\underline{v})$ below to ease the notation. In particular, we can rewrite our equation as the first order system 
\begin{align*}
\underline{K} v &= -\frac{1}{2\sqrt{2}}  G^R  \mathrm{Op} (i \xi) D_{\mathrm{m}}^{-1} v
+ \frac{1}{2\sqrt{2}}  G^R  \mathrm{Op} (i \xi) D_{\mathrm{m}}^{-1} \frac{1}{2\sqrt{2}} \left( G^R \mathrm{Op}(i \xi) D_\mathrm{m}^{-1} \frac{1}{\sqrt{2}}D_{\mathrm{m}}^{-1} \left( u + \underline{u} \right)\right) \\
& -\frac{1}{2\sqrt{2}} \left(  G D_\mathrm{m}^{-1}u  + G D_\mathrm{m}^{-1}\underline{u} \right)
-\frac{1}{2} \left[ G^R  \mathrm{Op} (i \xi) D_{\mathrm{m}}^{-1}, \underline{K}\right] D_\mathrm{m}^{-1} \frac{u + \underline{u}}{2}
\end{align*}
and 
\begin{align*}
K\underline{v} &= -\frac{1}{2\sqrt{2}}  G^R  \mathrm{Op} (i \xi) D_{\mathrm{m}}^{-1} \underline{v} 
+ \frac{1}{2\sqrt{2}}  G^R  \mathrm{Op} (i \xi) D_{\mathrm{m}}^{-1} \frac{1}{2\sqrt{2}} \left( G^R \mathrm{Op}(i \xi) D_\mathrm{m}^{-1} \frac{1}{\sqrt{2}}D_{\mathrm{m}}^{-1} \left( u + \underline{u} \right)\right) \\
& -\frac{1}{2\sqrt{2}} \left(  G D_\mathrm{m}^{-1}u  + G D_\mathrm{m}^{-1}\underline{u} \right)
-\frac{1}{2} \left[ G^R  \mathrm{Op} (i \xi) D_{\mathrm{m}}^{-1}, K\right] D_\mathrm{m}^{-1} \frac{u + \underline{u}}{2}.
\end{align*}
This takes the from 
\begin{eqnarray} \label{eq:KKbH}
\left( \begin{array}{cc} K & 0 \\
0 & \underline{K} 
\end{array} \right)
\left( \begin{array}{c} v\\
 \underline{v}
\end{array} \right)  = \left( \begin{array}{cc} H & 0 \\
0 & \underline{H} 
\end{array} \right)
\left( \begin{array}{c} v\\
 \underline{v}
\end{array} \right)   + R_{-1} \left( \begin{array}{c} v\\
 \underline{v}
\end{array} \right)
\end{eqnarray}
%
%
where $H=\underline{H}$ is a pseudo-differential operator of order $0$ defined by 
\begin{eqnarray} \label{def:H}
H= -\frac{1}{2\sqrt{2}}  G^R  \mathrm{Op} (i \xi) D_{\mathrm{m}}^{-1},
\end{eqnarray}
and  $R_{-1}$ is a matrix of pseudo-differential operators of order $-1$.
Since $H$ is a pseudo-differential operator depending quasi-periodically on $t$, we have $H:=\mathcal{H}(\omega t)$, where $\mathcal{H}:=\mathrm{Op}\,\mathfrak{h}(\varphi, x, \xi )$  is a pseudo-differential operator of order $0$.  We claim that there exists a $0$th order pseudo-differential operator,  depending quasi-periodically in $t$,  $M(\tau):=\mathcal{M}(\omega \tau)$, such that 
\begin{equation}\label{eq:Mop}
\mathcal{M}  \mathcal{K}  \mathcal{M}^{-1} = -i  \alpha \,\omega \cdot \partial_{\varphi} + D_\mathrm{m} - \mathcal{H} + \mathcal{Q}_{1}, \quad  \mathcal{K} := -i \alpha \, \omega \cdot \partial_{\varphi} + D_\mathrm{m},
\end{equation}
where  
\begin{align}\label{es:qrem}
\begin{split}
	\Vert \mathcal{Q}_1 \Vert_{-1, s, 0}^{\gamma, \mathcal{O}^{2\gamma}_\infty} &\lesssim   \Vert \mathcal{G}^R \Vert^{\gamma, \mathcal{O}^{2\gamma}_\infty}_{s+2\iota} , \\
	\Vert \mathcal{M}-\mathrm{Id}\Vert_{0, s, 0}^{\gamma, \mathcal{O}^{2\gamma}_\infty}  &\lesssim   \Vert \mathcal{G}^R \Vert^{\gamma, \mathcal{O}^{2\gamma}_\infty}_{s+2\iota}, \\ 
	\Vert (\mathcal{M}-\mathrm{Id})(v,\underline{v})\Vert_{s}^{\gamma, \mathcal{O}^{2\gamma}_\infty} &\lesssim   \Vert \mathcal{G}^R \Vert^{\gamma, \mathcal{O}^{2\gamma}_\infty}_{s+2\iota } \Vert (v, \underline{v})  \Vert^{\gamma, \mathcal{O}^{2\gamma}_\infty}_{s_0}+ \Vert \mathcal{G}^R \Vert^{\gamma, \mathcal{O}^{2\gamma}_\infty}_{s_0+2\iota } \Vert (v, \underline{v})  \Vert^{\gamma, \mathcal{O}^{2\gamma}_\infty}_{s}.
\end{split}
\end{align} 
and the same estimates hold for $\mathcal{M}^{-1}$ instead of $\mathcal{M}$.   Indeed, it follows from \cite[Lemma 9.6, Lemma 3.7]{berti2024reducibility} that there exists a symbol $d:=d(\varphi, x, \xi) \in S^0$, real valued, reversibility and parity preserving verifying 
\begin{eqnarray} \label{eq:d}
-  \alpha\, \omega \, .\partial_{\varphi} d + \xi  D_\mathrm{m}^{-1} (\xi) \partial_x d = \mathfrak{h} +r_{-2}, 
\end{eqnarray}
for some real valued, reversible and parity preserving  $r_{-2} \in S^{-2}$ and such that the estimates
\begin{align*}
\Vert d \Vert_{0, s, p}^{\gamma, \mathcal{O}^{2\gamma}_\infty}\lesssim  \Vert \mathcal{G}^R \Vert^{\gamma, \mathcal{O}^{2\gamma}_\infty}_{s+2\iota}, \quad 
\Vert r_{-2} \Vert_{-2, s, p}^{\gamma, \mathcal{O}^{2\gamma}_\infty}\lesssim  \Vert \mathcal{G}^R \Vert^{\gamma, \mathcal{O}^{2\gamma}_\infty}_{s+2\iota}.
\end{align*}
holds, for all $s \ge s_0$, $p \ge 0$. Defining $\mathcal{M}= \exp{ \mathrm{Op}(d)}$, we can then follow the proof of Proposition 9.5 in \cite{berti2024reducibility} to check that \eqref{eq:Mop}-\eqref{es:qrem} holds. More precisely, by a Lie expansion, we have 
\begin{eqnarray*}
\mathcal{M} \omega \cdot \partial_{\varphi} \mathcal{M}^{-1} = \omega   \cdot \partial_{\varphi}  - \mathrm{Op} \left(  \omega \cdot \partial_{\varphi} d \right) + \mathcal{Q}_1,
\end{eqnarray*}
where 
\begin{align*}
\mathcal{Q}_1:= \int_0^1 (1-\tau) \exp{\left( \mathrm{Op} (\tau d)\right)} \mathrm{ad}_{\mathrm{Op}(d)} \left[ \mathrm{Op} \left(  \omega \cdot \partial_{\varphi} d \right)\right] \exp{\left( -\mathrm{Op} (\tau d)\right) } d\tau \in S^{-1}	
\end{align*}
verifies 
\begin{align*}
\Vert \mathcal{Q}_1 \Vert_{-1, s, 0}^{\gamma, \mathcal{O}^{2\gamma}_\infty}\lesssim  \Vert \mathcal{G}^R \Vert^{\gamma, \mathcal{O}^{2\gamma}_\infty}_{s+2\iota}.	
\end{align*}
Similarly, we have 
\begin{eqnarray*}
\mathcal{M}D_\mathrm{m} \mathcal{M}^{-1} = D_\mathrm{m}  - \left[ D_\mathrm{m}, \mathrm{Op}(d) \right]  + \mathcal{Q}_2,
\end{eqnarray*}
where 
\begin{align*}
\mathcal{Q}_2:= \int_0^1 (1-\tau) \exp{\left( \mathrm{Op} (\tau d)\right)} \mathrm{ad}_{\mathrm{Op}(d)} \left[D_\mathrm{m}, \mathrm{Op}(d)  \right] \exp{\left( -\mathrm{Op} (\tau d)\right) } d\tau \in S^{-1},	
\end{align*}
verifies
\begin{align*}
	\Vert \mathcal{Q}_2 \Vert_{-1, s, 0}^{\gamma, \mathcal{O}^{2\gamma}_\infty}\lesssim  \Vert \mathcal{G}^R \Vert^{\gamma, \mathcal{O}^{2\gamma}_\infty}_{s+2\iota}
\end{align*}
and    
\begin{align*}
\left[ D_\mathrm{m}, \mathrm{Op}(d) \right] = \mathrm{Op}\left( i \xi D_\mathrm{m}^{-1} \partial_x d \right) + \mathcal{Q}_{3}, 	
\end{align*} 
with $\mathcal{Q}_{3} \in S^{-1}$ verifying
\begin{align*}
\Vert \mathcal{Q}_3 \Vert_{-1, s, 0}^{\gamma, \mathcal{O}^{2\gamma}_\infty}&\lesssim   \Vert \mathcal{G}^R \Vert^{\gamma, \mathcal{O}^{2\gamma}_\infty}_{s+2\iota}.
\end{align*} 
Collecting the identities from the Lie expansion, we see that \eqref{eq:Mop} holds. 
Thus, returning to \eqref{eq:KKbH}, it follows that 
\begin{align*}
 \mathcal{K}  \mathcal{M}^{-1}  = \mathcal{M}^{-1} \left(- i \alpha\,\omega \cdot \partial_{\varphi}  + D_{\mathrm{m}}  -\mathcal{H} \right) + R_{-1,2} , 
\end{align*}
where 
\begin{align*}
\Vert R_{-1,2}\Vert_{-1, s, 0}^{\gamma, \mathcal{O}^{2\gamma}_\infty} \lesssim   \Vert \mathcal{G}^R \Vert^{\gamma, \mathcal{O}^{2\gamma}_\infty}_{s+2\iota}.
\end{align*} 
Exchanging $v$ and $\underline{v}$, we define $\underline{\mathcal{M}}(\varphi,x ,\xi):=\overline{\mathcal{M}(\varphi,x ,-\xi)}$. We have $\underline{\mathcal{M}}=  \exp{\left( \mathrm{Op} \left( \underline{d}(\varphi,x,  \xi) \right)\right)}$ where  $ \underline{d}(\varphi,x,  \xi):= \overline{d(\varphi,x -\xi)} $ verifies the analogue of \eqref{eq:d} corresponding to $K$, that is
\begin{align*} 
-  \alpha\, \omega \, .\partial_{\varphi} \underline{d} - \xi  D_\mathrm{m}^{-1} (\xi) \partial_x \underline{d} =  \underline{\mathfrak{h}} +\underline{r_{-2}},  
\end{align*}
where $\underline{r_{-2}}(\varphi, x,\xi)= \overline{r_{-2}(\varphi, x,-\xi)}$ and 
\begin{align*}
	\underline{\mathfrak{h}}(\varphi, x,\xi)= \overline{\mathfrak{h}}(\varphi, x,-\xi)= \mathfrak{h}(\varphi, x,\xi),
\end{align*}
in view of the definition of $H$ and $\mathfrak{h}$, cf~\eqref{def:H}. 
We can thus define the transformations 
\begin{align*}
	 (v, \underline{v} ) \mapsto \Theta(\varphi) (v, \underline{v} ) :=  ( \mathcal{M}^{-1} v, \underline{\mathcal{M}}^{-1} \underline{v}), \quad \mathcal{T}:= \Theta \mathfrak{V}.
\end{align*}
Collecting the previous estimates and identities, it follows that $\mathcal{T}$ is real-to-real, parity and reversibility preserving and verifies the estimates 
\begin{align*}
\Vert \mathcal{T}-\mathrm{Id} \Vert^{\gamma, \mathcal{O}^{2\gamma}_\infty}_{0,s,0} &\lesssim
	 \Vert \mathcal{G}^R \Vert^{\gamma, \mathcal{O}^{2\gamma}_\infty}_{s+2\iota}\left(1+\Vert \mathcal{G}^R \Vert^{\gamma, \mathcal{O}^{2\gamma}_\infty}_{s_0 } \right) 
	+ \Vert \mathcal{G}^R \Vert^{\gamma, \mathcal{O}^{2\gamma}_\infty}_{s_0+2\iota}\left( 1+ \Vert \mathcal{G}^R \Vert^{\gamma, \mathcal{O}^{2\gamma}_\infty}_{s }\right), \\
\Vert (\mathcal{T}-\mathrm{Id})  (u,\underline{u}) \Vert^{\gamma, \mathcal{O}^{2\gamma}_\infty}_s &\lesssim  \Vert \mathcal{G}^R \Vert^{\gamma, \mathcal{O}^{2\gamma}_\infty}_{s +2\iota }\left(
	1 + \Vert \mathcal{G}^R \Vert^{\gamma, \mathcal{O}^{2\gamma}_\infty}_{s_0 +2\iota } 
	\right)\Vert  (u,\underline{u})  \Vert^{\gamma, \mathcal{O}^{2\gamma}_\infty}_{s_0}
	+\Vert \mathcal{G}^R \Vert^{\gamma, \mathcal{O}^{2\gamma}_\infty}_{s_0 +2\iota } \left(1+\Vert \mathcal{G}^R \Vert^{\gamma, \mathcal{O}^{2\gamma}_\infty}_{s_0+2\iota } \right)\Vert  (u,\underline{u})  \Vert^{\gamma, \mathcal{O}^{2\gamma}_\infty}_{s} .
\end{align*}
Indeed,  we use Lemma \ref{lem:pseudocom} together with \eqref{es:qrem} and \eqref{EstimatesR} to obtain
\begin{align*}
	\Vert \mathcal{T}-\mathrm{Id} \Vert^{\gamma, \mathcal{O}^{2\gamma}_\infty}_{0,s,0} &=
	\Vert \Theta \mathfrak{V}-\mathrm{Id} \Vert^{\gamma, \mathcal{O}^{2\gamma}_\infty}_{0,s,0} 
	=
	\Vert (\Theta-\mathrm{Id}) \mathfrak{V} +  \mathfrak{V}-\mathrm{Id} \Vert^{\gamma, \mathcal{O}^{2\gamma}_\infty}_{0,s,0} \\
	&\leq 
	\Vert (\Theta-\mathrm{Id}) \mathfrak{V}  \Vert^{\gamma, \mathcal{O}^{2\gamma}_\infty}_{0,s,0} +\Vert   \mathfrak{V}-\mathrm{Id} \Vert^{\gamma, \mathcal{O}^{2\gamma}_\infty}_{0,s,0} \\
	&\lesssim
	\Vert  \Theta-\mathrm{Id}  \Vert^{\gamma, \mathcal{O}^{2\gamma}_\infty}_{0,s,0} \Vert  \mathfrak{V}  \Vert^{\gamma, \mathcal{O}^{2\gamma}_\infty}_{0,s_0,0}
	+\Vert  \Theta-\mathrm{Id}   \Vert^{\gamma, \mathcal{O}^{2\gamma}_\infty}_{0,s_0,0}\Vert   \mathfrak{V}  \Vert^{\gamma, \mathcal{O}^{2\gamma}_\infty}_{0,s,0} +\Vert   \mathfrak{V}-\mathrm{Id} \Vert^{\gamma, \mathcal{O}^{2\gamma}_\infty}_{0,s,0}  \\
	&\lesssim
	 \Vert \mathcal{G}^R \Vert^{\gamma, \mathcal{O}^{2\gamma}_\infty}_{s+2\iota}\Vert  \mathfrak{V}  \Vert^{\gamma, \mathcal{O}^{2\gamma}_\infty}_{0,s_0,0}
	+ \Vert \mathcal{G}^R \Vert^{\gamma, \mathcal{O}^{2\gamma}_\infty}_{s_0+2\iota}\Vert   \mathfrak{V}  \Vert^{\gamma, \mathcal{O}^{2\gamma}_\infty}_{0,s,0} +\Vert   \mathfrak{V}-\mathrm{Id} \Vert^{\gamma, \mathcal{O}^{2\gamma}_\infty}_{0,s,0}  \\
	&\lesssim
	 \Vert \mathcal{G}^R \Vert^{\gamma, \mathcal{O}^{2\gamma}_\infty}_{s+2\iota}\left(1+\Vert \mathcal{G}^R \Vert^{\gamma, \mathcal{O}^{2\gamma}_\infty}_{s_0} \right) 
	+ \Vert \mathcal{G}^R \Vert^{\gamma, \mathcal{O}^{2\gamma}_\infty}_{s_0+2\iota}\left( 1+ \Vert \mathcal{G}^R \Vert^{\gamma, \mathcal{O}^{2\gamma}_\infty}_s\right)+\Vert \mathcal{G}^R \Vert^{\gamma, \mathcal{O}^{2\gamma}_\infty}_s  \\
	&\lesssim
	 \Vert \mathcal{G}^R \Vert^{\gamma, \mathcal{O}^{2\gamma}_\infty}_{s+2\iota}\left(1+\Vert \mathcal{G}^R \Vert^{\gamma, \mathcal{O}^{2\gamma}_\infty}_{s_0 } \right) 
	+ \Vert \mathcal{G}^R \Vert^{\gamma, \mathcal{O}^{2\gamma}_\infty}_{s_0+2\iota}\left( 1+ \Vert \mathcal{G}^R \Vert^{\gamma, \mathcal{O}^{2\gamma}_\infty}_{s }\right) , \\
	\Vert (\mathcal{T}-\mathrm{Id}) (u,\underline{u}) \Vert^{\gamma, \mathcal{O}^{2\gamma}_\infty}_{s} &=
	\Vert ( \Theta \mathfrak{V}-\mathrm{Id}) (u,\underline{u})  \Vert^{\gamma, \mathcal{O}^{2\gamma}_\infty}_{s} 
	=
	\Vert (\Theta-\mathrm{Id}) \mathfrak{V} (u,\underline{u})  + ( \mathfrak{V}-\mathrm{Id}) (u,\underline{u})  \Vert^{\gamma, \mathcal{O}^{2\gamma}_\infty}_{s} \\
	&\leq 
	\Vert (\Theta-\mathrm{Id}) \mathfrak{V}  (u,\underline{u})  \Vert^{\gamma, \mathcal{O}^{2\gamma}_\infty}_{s} +\Vert  ( \mathfrak{V}-\mathrm{Id}) (u,\underline{u})  \Vert^{\gamma, \mathcal{O}^{2\gamma}_\infty}_{s} \\
	&\lesssim  \Vert \mathcal{G}^R \Vert^{\gamma, \mathcal{O}^{2\gamma}_\infty}_{s+2\iota } \Vert \mathfrak{V} (u,\underline{u})  \Vert^{\gamma, \mathcal{O}^{2\gamma}_\infty}_{s_0}+ \Vert \mathcal{G}^R \Vert^{\gamma, \mathcal{O}^{2\gamma}_\infty}_{s_0+2\iota } \Vert \mathfrak{V} (u,\underline{u}) \Vert^{\gamma, \mathcal{O}^{2\gamma}_\infty}_{s}+\Vert  ( \mathfrak{V}-\mathrm{Id}) (u,\underline{u})  \Vert^{\gamma, \mathcal{O}^{2\gamma}_\infty}_{s} \\
	&=  \Vert \mathcal{G}^R \Vert^{\gamma, \mathcal{O}^{2\gamma}_\infty}_{s+2\iota } \Vert (\mathrm{Id}+R) (u,\underline{u})  \Vert^{\gamma, \mathcal{O}^{2\gamma}_\infty}_{s_0}+ \Vert \mathcal{G}^R \Vert^{\gamma, \mathcal{O}^{2\gamma}_\infty}_{s_0+2\iota } \Vert (\mathrm{Id}+R) (u,\underline{u}) \Vert^{\gamma, \mathcal{O}^{2\gamma}_\infty}_{s}+\Vert  R (u,\underline{u})  \Vert^{\gamma, \mathcal{O}^{2\gamma}_\infty}_{s} \\
	&\leq  \Vert \mathcal{G}^R \Vert^{\gamma, \mathcal{O}^{2\gamma}_\infty}_{s+2\iota } \left( \Vert   (u,\underline{u})  \Vert^{\gamma, \mathcal{O}^{2\gamma}_\infty}_{s_0}    + \Vert R (u,\underline{u})  \Vert^{\gamma, \mathcal{O}^{2\gamma}_\infty}_{s_0} \right)  + \Vert \mathcal{G}^R \Vert^{\gamma, \mathcal{O}^{2\gamma}_\infty}_{s_0+2\iota } \left(\Vert   (u,\underline{u})  \Vert^{\gamma, \mathcal{O}^{2\gamma}_\infty}_{s}+ \Vert R (u,\underline{u}) \Vert^{\gamma, \mathcal{O}^{2\gamma}_\infty}_{s} \right)  \\
	&+\Vert  R (u,\underline{u})  \Vert^{\gamma, \mathcal{O}^{2\gamma}_\infty}_{s}  \\
	&\lesssim \Vert \mathcal{G}^R \Vert^{\gamma, \mathcal{O}^{2\gamma}_\infty}_{s+2\iota } \left( 1+  \Vert \mathcal{G}^R \Vert^{\gamma, \mathcal{O}^{2\gamma}_\infty}_{s_0}\Vert  (u,\underline{u})  \Vert^{\gamma, \mathcal{O}^{2\gamma}_\infty}_{s_0}  \right)  \\
	&+ \Vert \mathcal{G}^R \Vert^{\gamma, \mathcal{O}^{2\gamma}_\infty}_{s_0+2\iota } \left( 1  +  \Vert \mathcal{G}^R \Vert^{\gamma, \mathcal{O}^{2\gamma}_\infty}_s \Vert  (u,\underline{u})  \Vert^{\gamma, \mathcal{O}^{2\gamma}_\infty}_{s_0}+ \Vert \mathcal{G}^R \Vert^{\gamma, \mathcal{O}^{2\gamma}_\infty}_{s_0} \Vert  (u,\underline{u})  \Vert^{\gamma, \mathcal{O}^{2\gamma}_\infty}_{s}\right) \\
	& + \Vert \mathcal{G}^R \Vert^{\gamma, \mathcal{O}^{2\gamma}_\infty}_s \Vert  (u,\underline{u})  \Vert^{\gamma, \mathcal{O}^{2\gamma}_\infty}_{s_0}+ \Vert \mathcal{G}^R \Vert^{\gamma, \mathcal{O}^{2\gamma}_\infty}_{s_0} \Vert  (u,\underline{u})  \Vert^{\gamma, \mathcal{O}^{2\gamma}_\infty}_{s} \\ 
	& \lesssim  \Vert \mathcal{G}^R \Vert^{\gamma, \mathcal{O}^{2\gamma}_\infty}_{s +2\iota }\left(
	1 + \Vert \mathcal{G}^R \Vert^{\gamma, \mathcal{O}^{2\gamma}_\infty}_{s_0 +2\iota } 
	\right)\Vert  (u,\underline{u})  \Vert^{\gamma, \mathcal{O}^{2\gamma}_\infty}_{s_0}
	+\Vert \mathcal{G}^R \Vert^{\gamma, \mathcal{O}^{2\gamma}_\infty}_{s_0 +2\iota } \left(1+\Vert \mathcal{G}^R \Vert^{\gamma, \mathcal{O}^{2\gamma}_\infty}_{s_0+2\iota } \right)\Vert  (u,\underline{u})  \Vert^{\gamma, \mathcal{O}^{2\gamma}_\infty}_{s} .
\end{align*}
Moreover, one has
\begin{align*} 
 \left( \begin{array}{cc} K & 0 \\
0 & \underline{K} 
\end{array} \right) \mathcal{T}(\omega \tau) 
\left( \begin{array}{c} u\\
 \underline{u}
\end{array} \right) &=  
 \left( \begin{array}{cc} K & 0 \\
0 & \underline{K} 
\end{array} \right)\Theta(\omega \tau) \left( \mathfrak{V} 
\left( \begin{array}{c} u\\
 \underline{u}
\end{array}\right) \right)   \\
&=  \left( \begin{array}{cc} K & 0 \\
0 & \underline{K} 
\end{array} \right)\Theta(\omega \tau)   
\left( \begin{array}{c} v\\
 \underline{v}
\end{array}\right)    \\
&=  \left( \begin{array}{cc} K & 0 \\
0 & \underline{K} 
\end{array} \right)   
\left( \begin{array}{c} M^{-1} v\\
\underline{M}^{-1} \underline{v}
\end{array}\right)    \\
&=    
\left( \begin{array}{c} KM^{-1} v\\
 \underline{K} \underline{M}^{-1} \underline{v}
\end{array}\right)    \\
&=    
\left( \begin{array}{c} M^{-1}  \left(- i \alpha \partial_{\tau}  + D_{\mathrm{m}}  -H \right)v + R_{-1,2}v\\
 \underline{M}^{-1}  \left(  i \alpha \partial_{\tau}  + D_{\mathrm{m}}  - \underline{H} \right)v +  \underline{ R}_{-1,2}v
\end{array}\right)    \\
&=    
\left( \begin{array}{c}  R_{-1,3}v\\
 \underline{R}_{-1,3} \underline{v}
\end{array}\right)
:=  
\left( \begin{array}{c}  R_{-1,4}u\\
 \underline{R}_{-1,4} \underline{u}
\end{array}\right) ,
\end{align*} 
for some $R_{-1,4}=\mathrm{Op}(r_{-1,4})(\omega \tau; \cdot)$ verifying 
\begin{align*}
\Vert R_{-1,4} \Vert_{-1, s, 0}^{\gamma, \mathcal{O}^{2\gamma}_\infty}&\lesssim  \left(\Vert \mathcal{G}^R \Vert^{\gamma, \mathcal{O}^{2\gamma}_\infty}_{s +2\iota }+
\Vert \mathcal{G} \Vert^{\gamma, \mathcal{O}^{2\gamma}_\infty}_{s} \right)\left(
	1 + \Vert \mathcal{G}^R \Vert^{\gamma, \mathcal{O}^{2\gamma}_\infty}_{s_0 +2\iota } 
	\right)\Vert  (u,\underline{u})  \Vert^{\gamma, \mathcal{O}^{2\gamma}_\infty}_{s_0-1} \\
	& \quad  +\left(\Vert \mathcal{G}^R \Vert^{\gamma, \mathcal{O}^{2\gamma}_\infty}_{s +2\iota }+
\Vert \mathcal{G} \Vert^{\gamma, \mathcal{O}^{2\gamma}_\infty}_{s} \right) \left(1+\Vert \mathcal{G}^R \Vert^{\gamma, \mathcal{O}^{2\gamma}_\infty}_{s_0+2\iota } \right)\Vert  (u,\underline{u})  \Vert^{\gamma, \mathcal{O}^{2\gamma}_\infty}_{s-1}  .
\end{align*} 
Using the definition of $K$ and $\underline{K}$, this is equivalent to \eqref{eq:nofur} and thus concludes the proof of Proposition \ref{prop:reducall1}.
\end{proof}

  \bibliographystyle{plain}
 \bibliography{References_Bibtex}

\end{document}